\documentclass[final,hidelinks,onefignum,onetabnum]{siamart220329}

\usepackage{lipsum}
\usepackage{amsfonts}
\usepackage{graphicx}
\usepackage{epstopdf}
\ifpdf
\DeclareGraphicsExtensions{.eps,.pdf,.png,.jpg}
\else
\DeclareGraphicsExtensions{.eps}
\fi
\usepackage{booktabs}
\usepackage{hyperref}
\usepackage{stmaryrd}
\usepackage{bm}
\usepackage{caption}
\usepackage{subcaption}
\usepackage{multirow}
\usepackage{enumitem}
\usepackage{tikz}
\usepackage[ruled]{algorithm2e}
\newtheorem{example}{\textbf{Example}}[section]
\usepackage{CJKutf8}
\usepackage{amsmath, amssymb, cases, extarrows, float, xcolor, titlesec, iitem, colortbl, setspace, cite}  
\usepackage{geometry}

\newsiamthm{assumption}{Assumption}

\newtheorem{remark}[theorem]{Remark}
\newtheorem{notion}{Notion}[section]

\headers{}{}

\title{Towards $C^0$ finite element methods for fourth-order elliptic equation. Part I: general boundary conditions}

\author{Xihao Zhang\thanks{School of Mathematics and Computational Science, Xiangtan University, Xiangtan 411105, P.R.China (\email{202431510169@smail.xtu.edu.cn})}
\and Hengguang Li\thanks{Department of Mathematics and Institute for AI and Data Science, Wayne State University, Detroit, Michigan, USA (\email{li@wayne.edu}).}
\and Nianyu Yi\thanks{Hunan Key Laboratory for Computation and Simulation in Science and Engineering; School of Mathematics and Computational Science, Xiangtan University, Xiangtan 411105, P.R.China (\email{yinianyu@xtu.edu.cn}).}
\and Peimeng Yin\thanks{Corresponding author. Department of Mathematical Sciences, The University of Texas at El Paso, El Paso, TX 79968, USA (\email{pyin@utep.edu}).}} 

\usepackage{amsopn}

\newcommand{\thm}[1]{\begin{theorem}#1\end{theorem}}
\newcommand{\lem}[1]{\begin{lemma}#1\end{lemma}}
\newcommand{\dfn}[1]{\begin{definition}#1\end{definition}}

\newcommand{\exa}[1]{\begin{example}#1\end{example}}
\newcommand{\prf}[1]{\begin{proof}#1\end{proof}}

\newcommand{\seqn}[1]{\begin{subequations}#1\end{subequations}}
\newcommand{\al}[1]{\begin{align*}#1\end{align*}}
\newcommand{\ali}[1]{\begin{align}#1\end{align}}

\newcommand{\ca}[1]{\begin{cases}#1\end{cases}}
\newcommand{\nca}[1]{\begin{numcases}#1\end{numcases}}
\newcommand{\pict}[1]{\begin{tikzpicture}#1\end{tikzpicture}}

\newcommand{\fig}[1]{\begin{figure*}#1\end{figure*}}
\newcommand{\bbr}[1]{\left\{#1\right\}}

\newcommand{\sbr}[1]{\left(#1\right)}

\numberwithin{equation}{section}
\hypersetup{
	hidelinks, 
	colorlinks=true,
	linkcolor=blue,
	citecolor=black
}
\ifpdf
\hypersetup{
  pdftitle={},
  pdfauthor={}
}
\fi

\usepackage{comment}

\begin{document}
\begin{CJK}{UTF8}{gbsn}

\maketitle

\begin{abstract}
This paper is part of a series developing $C^0$ finite element methods for fourth-order elliptic equations on polygonal domains. Here, we investigate how boundary conditions influence the design of effective $C^0$ schemes, specifically focusing on equations without lower-order terms, namely the biharmonic equation. We propose a modified mixed formulation that decomposes the problem into a system of Poisson equations, where the number of equations depends on both the largest interior angle and the boundary conditions on its two adjacent sides. In contrast to the naive mixed formulation, which involves only two Poisson problems, the proposed approach guarantees convergence to the true solution for arbitrary polygonal domains and general boundary conditions, including Navier, Neumann, and mixed boundary conditions. $C^0$ finite element algorithms are developed, rigorous error estimates are established, and numerical experiments are presented to demonstrate the well-posedness and effectiveness of the proposed method.
\end{abstract}




\section{Introduction}

Let $\Omega \subset \mathbb{R}^2$ be a polygonal domain whose boundary is decomposed as
\(
\partial \Omega = \overline{\Gamma}_D \cup \overline{\Gamma}_N,
\)
where $\Gamma_D$ and $\Gamma_N$ are two disjoint open subsets of $\partial \Omega$.
Given $f \in L^2(\Omega)$, we are interested in developing $C^0$ finite element algorithms for the following fourth-order elliptic problem:
\begin{align}\label{eqn-initial-}
\begin{cases}
\Delta^2 u - b \Delta u + c u = f, & \text{in } \Omega, \\
u = \Delta u = 0, & \text{on } \Gamma_D, \\
\partial_{\mathbf n} u = \partial_{\mathbf n} \Delta u = 0, & \text{on } \Gamma_N.
\end{cases}
\end{align}
Here, $b$ and $c$ are nonnegative constants, and $\mathbf n$ denotes the outward unit normal vector on $\partial \Omega$. The boundary segments $\Gamma_D$ and $\Gamma_N$ correspond to Dirichlet and Neumann boundary conditions, respectively.

To develop $C^0$ finite element methods for the biharmonic equation with Navier boundary conditions, a typical approach is to reformulate the problem as a fully decoupled system of lower-order elliptic equations, which allows for the direct use of $C^0$ finite element spaces \cite{li2023ac}. However, due to the presence of lower-order terms in \eqref{eqn-initial-}, the equation cannot be rewritten as a fully decoupled system. Consequently, existing $C^0$ finite element methods with correction techniques for the biharmonic equation under Navier boundary conditions cannot be directly extended to problem \eqref{eqn-initial-}.

In a series of works, we aim to develop $C^0$ finite element methods for problem \eqref{eqn-initial-}. In the present paper, we first focus on the influence of boundary conditions on $C^0$ finite element methods for biharmonic problem, namely the case $b = c = 0$ in \eqref{eqn-initial}:
\begin{align}\label{eqn-initial}
\begin{cases}
\Delta^2 u = f, & \text{in } \Omega, \\
u = \Delta u = 0, & \text{on } \Gamma_D, \\
\partial_{\mathbf n} u = \partial_{\mathbf n} \Delta u = 0, & \text{on } \Gamma_N.
\end{cases}
\end{align}

The biharmonic problem has been extensively studied in fields such as fluid mechanics and structural mechanics \cite{rachh2020solution, mardanov2016solution, timoshenko1959theory, ugural2009stresses, sokolnikoff1946mathematical}, where different types of boundary conditions correspond to specific application scenarios.
For the boundary conditions defined in problem~\eqref{eqn-initial}, when $\Gamma_D = \partial\Omega$, the problem is said to satisfy Navier boundary conditions \cite{MR1422248, MR2499118}. This case arises in structural mechanics, for instance in models describing the static loading of a thin plate with hinged boundaries.
In contrast, when $\Gamma_N = \partial\Omega$, the problem is referred to as the Neumann boundary condition case \cite{matevossian2019biharmonic}. Problems of this type are associated with physical models such as the transverse vibration of a thin plate with free boundaries and the natural deformation of membrane materials \cite{courant2008methods}.

In this paper, we treat the cases $\Gamma_D = \partial\Omega$ and $\Gamma_N = \partial\Omega$ as trival cases of mixed boundary conditions.
In general, $\Gamma_D$ and $\Gamma_N$ may be nontrivial subsets of $\partial\Omega$, thereby encompassing a broader class of physical models.
For such fourth-order problems, finite element methods based on direct variational formulations can be employed \cite{brenner2008mathematical}. Typical choices include $C^1$ finite elements, such as the Argyris element, which consists of fifth-degree polynomials \cite{taibaouic1}, and the Morley element, a quadratic nonconforming element \cite{li2018adaptive}.
However, the construction of these finite element spaces is often complicated, making the corresponding methods challenging to implement in practice. 
Alternative approaches, including discontinuous Galerkin methods \cite{georgoulis2009discontinuous} and the $C^0$ interior penalty Galerkin method \cite{brenner2015c, cao2025posteriori}, are also viable. 
Nevertheless, the choice of penalty parameters is often subtle.
For the biharmonic equation with Navier or Dirichlet boundary conditions, the problem can be decomposed into a system of Poisson equations and a Stokes equation \cite{gallistl2017stable, li2025analysis}. However, such a decomposition is no longer applicable when the fourth-order elliptic equation includes lower-order terms.

Motivated by the structure of the boundary conditions in problem~\eqref{eqn-initial}, an intuitive strategy is to decouple the biharmonic problem into a system of Poisson equations, which can then be efficiently solved using standard finite element methods.
Mixed finite element methods based on $C^0$ finite elements offer an attractive framework for the numerical approximation of the biharmonic problem~\eqref{eqn-initial} \cite{ciarlet1974mixed, ciarlet1975dual, falk1978approximation, ciarlet2002finite, cheng2000some, monk1987mixed, babuvska1980analysis, li2023ac, li2025analysis}. A large portion of the existing literature focuses on clamped boundary conditions (also referred to as Dirichlet boundary conditions), namely $(u = \partial_{\mathbf n} u = 0 \text{ on } \partial\Omega)$ \cite{hu2021family, destuynder2013mathematical, rafetseder2018decomposition, li2025analysis}.
For biharmonic problems with Navier boundary conditions, it has been observed \cite{nazarov2007boundary, zhang2008invalidity, gerasimov2012corners, de2019comparing, li2023ac} that when the domain $\Omega$ is nonconvex, the mixed finite element solution of problem~\eqref{eqn-initial} may fail to coincide with the true solution for certain source terms $f \in L^2(\Omega)$. This discrepancy is known as the Sapongyan paradox \cite{MR2270884, zhang2008invalidity}.
This arises because the function space of the solution to a mixed variational problem may not coincide with that of the corresponding direct variational problem, owing to the presence of a low-regularity term that does not belong to $H^2(\Omega)$ \cite{nazarov1994elliptic, zhang2008invalidity}.

Recently, a modified mixed formulation was proposed in \cite{li2023ac} for the biharmonic equation with Navier boundary conditions. This formulation decomposes the problem into a system of three Poisson equations, in which an additional intermediate Poisson problem is introduced to ensure that the solution is confined to the correct function space.
A similar phenomenon was further observed for the triharmonic equation with simply supported boundary conditions \cite{li2025ac}, which can be naively decomposed into three Poisson equations with Dirichlet boundary conditions. However, the solutions obtained from these Poisson equations may still lie in a larger function space, resulting in weak solutions that exhibit behavior similar to the Sapongyan paradox. To address this issue, one to three additional Poisson equations, depending on the largest interior angle, may be introduced to further confine the solution to the correct function space.

Existing \(C^0\) finite element algorithms for the biharmonic equation are largely restricted to Navier boundary conditions. To develop an effective \(C^0\) finite element algorithm for \eqref{eqn-initial-}, it is crucial to investigate how general boundary conditions influence the convergence of the \(C^0\) finite element approximation. From this perspective, the Navier case can be viewed as a special instance.
By analyzing the orthogonal complement of the range of the Laplace operator mapping $H^2$, equipped with appropriate boundary conditions, into $L^2(\Omega)$, we first construct a modified decoupled scheme for problem~\eqref{eqn-initial}. This scheme transforms the biharmonic problem into a system of three or four Poisson problems, in which the number of equations depends on the largest interior angle of the domain (possibly including angles greater than $\frac{\pi}{2}$), as well as on the specific choice of boundary conditions.
In addition to the two Poisson problems directly arising from the biharmonic problem, a varying number of additional Poisson problems is required to ensure that the solution is confined to the correct Sobolev space. In particular, if only one interior angle exceeds \(\frac{\pi}{2}\), then at most one additional Poisson problem is needed for Navier or Neumann boundary conditions, whereas up to two additional Poisson problems are required for mixed boundary conditions.
A theoretical analysis is presented to demonstrate the equivalence between the proposed mixed formulation, that is, the system of Poisson problems, and the original biharmonic problem in appropriate Sobolev spaces.
Each of these Poisson problems can be solved using standard finite element methods, thereby forming a hybrid finite element algorithm for the biharmonic problem \eqref{eqn-initial}. In this way, the numerical solution of problem \eqref{eqn-initial} via the resulting system of Poisson problems becomes both simple and efficient. 

To solve the proposed mixed formulation, we develop a numerical algorithm using piecewise linear \(C^0\) finite elements. In addition, we provide an error analysis for the finite element approximations of both the auxiliary function \(w\) and the solution \(u\).
For the auxiliary function \(w\), the error in the \(H^1\) norm is standard and converges at a rate of \(h^{\alpha}\) on a quasi-uniform mesh, where \(\alpha\) depends on the largest interior angle and the boundary conditions on both sides of that angle.
The \(L^2\) error estimate is derived using a standard duality argument.
For the solution $u$, the approximation has a convergence rate of $h^{\min\{1, 2\alpha\}}$ in general. For Navier and Neumann boundary conditions, the optimal convergence rate is $h^1$, whereas for mixed boundary conditions, the convergence rate is $h^1$ if the interior angle is no more than $\pi$ and $h^{2\alpha}$ if the angle is greater than $\pi$.
Numerical examples are presented to compare the proposed method with existing numerical approaches, such as the $C^0$ interior penalty method. The solutions obtained from the proposed $C^0$ finite element algorithm agree very well with those computed using the $C^0$ interior penalty method \cite{brenner2015c, cao2025posteriori}. The observed numerical convergence rates are in good agreement with the theoretical predictions.

The remainder of the paper is organized as follows.
In \Cref{section-2}, drawing on the theory for second-order elliptic equations developed in \cite{grisvard1992singularities}, we identify the deficiencies of the solution space associated with the naive decoupled formulation. We then characterize the dimension of the orthogonal complement of the image of the Laplace operator and construct a corresponding basis. Based on these results, we propose a modified scheme and establish its well-posedness.
In \Cref{section-3}, we present a mixed finite element algorithm for the modified scheme and derive error estimates for both the primary solution and the auxiliary function on quasi-uniform meshes.
The extension to Neumann boundary conditions is discussed in \Cref{section-4}.
Finally, in \Cref{section-5}, we report numerical experiments on a range of representative polygonal domains and boundary conditions to verify the effectiveness of the proposed scheme.

\section{Biharmonic problem with mixed boundary conditions}\label{section-2}


Let $\gamma=(\gamma_1, \ldots, \gamma_d)\in\mathbb Z^d_{\geq0}$ be a multi-index, and define  
\(\partial^\gamma := \partial_{x_1}^{\gamma_1}\cdots\partial_{x_d}^{\gamma_d}\), 
\(|\gamma| := \sum_{i=1}^d \gamma_i\).
For $m\geq 0$, the Sobolev space $H^m(\Omega)$ consists of functions whose weak derivatives up to order $m$ are square-integrable. In particular,  
\(L^2(\Omega) := H^0(\Omega)\).
For $s>0$, let $s=m+t$ with $m\in \mathbb Z_{\geq 0}$ and $0<t<1$. For $D\subseteq \mathbb{R}^d$, the fractional Sobolev space $H^s(D)$ consists of distributions $v$ in $D$ such that  
\[
\|v\|^2_{H^s(D)} := \|v\|^2_{H^m(D)} 
+ \sum_{|\gamma|= m}\int_{D}\int_{D} \frac{|\partial^\gamma v(x) - \partial^\gamma v(y)|^2 }{|x-y|^{d+2t}} \,dx\,dy < \infty .
\]
The norm $\|\cdot\|_{H^s(\Omega)}$ is abbreviated as $\|\cdot\|_s$ for $s>0$, and $\|\cdot\|_{L^2(\Omega)}$ is abbreviated as $\|\cdot\|$.
We define $H_0^s(D)$ as the closure of $C_0^\infty(D)$ in $H^s(D)$, 
$H^{-s}(D)$ as the dual space of $H_0^s(D)$, 
$\widetilde{H}^s(D)$ as the space of all functions $v$ on $D$ whose extension $\tilde{v}$ by zero outside $D$ belongs to $H^s(\mathbb{R}^d)$, and $\widetilde{H}^{-s}(D)$ as the dule space of $\widetilde{H}^s(D)$.
In addition, we use the notation $a \lesssim b$ for quantities $a,b$ to mean $a \leq C b$,
where $C>0$ is a constant depending only on $\Omega$ and the boundary conditions.

In this section, we consider problem \eqref{eqn-initial} under the assumption that $\Gamma_D \neq \emptyset$. The special case where $\Gamma_N = \partial\Omega$ will be addressed in \Cref{sec-neumann}.
We define the Sobolev space
\begin{align}
V^2 := \{v\in H^2(\Omega):v|_{\Gamma_D}=0, \ \partial_{\mathbf n}v|_{\Gamma_N}=0\}. \label{dfn-V2}
\end{align}
The variational formulation of problem \eqref{eqn-initial} is then: find $u\in V^2$ such that
\begin{align} \label{eqn-var}
a(u,v) = (f,v), \quad \forall v \in V^2,
\end{align}
where the bilinear form $a(u,v)$ is defined by
\begin{equation}
    a(u,v):=\int_\Omega\Delta u\Delta v\,\mathrm dx.
\end{equation}

For any \( u, v \in V^2 \), the bilinear form \( a(u,v) \) is continuous on \( V^2 \):
\begin{align*}
	|a(u,v)|
	&= \left|\int_\Omega \Delta u \, \Delta v \, \mathrm{d}x\right|
	\leq \|\Delta u\| \, \|\Delta v\|
	\leq \|u\|_2 \, \|v\|_2.
\end{align*}
For any \( u \in V^2 \), the coercivity of \( a(u,u) \) follows from the inequality \cite[Theorem 2.2.3]{grisvard1992singularities}:
\begin{align*}
	\|u\|_2^2 \lesssim \|\Delta u\|^2 = a(u,u).
\end{align*}
Therefore, the Lax--Milgram Theorem guarantees that the variational problem \eqref{eqn-var} is well-posed.

\subsection{Naive mixed formulation}\label{sec2.1}
The problem \eqref{eqn-initial} can be decoupled into two Poisson problems with mixed boundary conditions
\ali{\label{eqn-dec}
    \ca{
	-\Delta w = f \ &\text{in }\Omega,\\
	w = 0 \ &\text{on }\Gamma_D,\\
	\partial_{\mathbf n} w = 0 \ &\text{on }\Gamma_N,
    }\qquad \text{and} \qquad
    \ca{
        -\Delta\bar u = w \ &\text{in }\Omega,\\
	\bar u = 0 \ &\text{on }\Gamma_D,\\
	\partial_{\mathbf n} \bar u=0 \ &\text{on }\Gamma_N,
    }
}
which is referred to as the naive mixed formulation. Define the Sobolev space
\al{
	V^1&:=\{v\in H^1(\Omega):v|_{\Gamma_D}=0\},
}
Then the corresponding variational form for \eqref{eqn-dec} is to find \( \bar{u}, w \in V^1 \) such that
\ali{\label{eqn-dec-var}
    \ca{
    A(w,\phi)=(f,\phi), \quad\forall\phi\in V^1,\\
    A(\bar u,\psi)=(w,\psi), \quad\forall\psi\in V^1,
    }
}
where the bilinear form
\begin{equation}
    A(\phi,\psi):=\int_\Omega\nabla\phi\cdot\nabla\psi\, \mathrm dx.
\end{equation}

This decoupled system consists of two standard Poisson problems, each of which is well posed and admits a unique solution under mixed boundary conditions, provided that $\Gamma_D \neq \emptyset$. Given $f \in H^{-1}(\Omega)$, the variational problem \eqref{eqn-dec-var} admits a weak solution $(\bar u, w) \in V^1 \subset H^1(\Omega)$. Since the solution $u$ to the original problem \eqref{eqn-initial} belongs to $V^2$, it is natural to question whether the solution $\bar u$ from the mixed formulation \eqref{eqn-dec} also lies in $V^2$. If not, then $\bar u$ cannot be regarded as the true solution to the original biharmonic problem \eqref{eqn-initial}.

Navier boundary conditions, i.e., $\Gamma_D = \partial \Omega$, the equivalence between the mixed formulation \eqref{eqn-dec} and the original biharmonic problem has been studied in \cite{li2023ac}. It was shown that, in domains with reentrant corners, the solution of the mixed formulation \eqref{eqn-dec} does not belong to the same Sobolev space as the solution of the biharmonic problem \eqref{eqn-initial}, a phenomenon known as the Sapongyan paradox \cite{nazarov1994elliptic, zhang2008invalidity}. 
For biharmonic problems with mixed boundary conditions as in \eqref{eqn-initial}, a similar discrepancy also arises for the solution of the mixed formulation \eqref{eqn-dec}. In the following, we analyze the solution structure in detail and develop effective methods to address this issue.

\subsection{Modified mixed formulation}
%
%
For the biharmonic problem \eqref{eqn-initial}, we allow interior angles equal to $\pi$ at points where different types of boundary conditions meet, and we regard each such boundary transition point as a vertex of the polygonal domain $\Omega$ for analytical purposes. Then each edge of $\Omega$ lies entirely in either $\bar\Gamma_D$ or $\bar\Gamma_N$.

	

\dfn{[Vertex, edge, and indicator sets]\label{dfn-idx}
	Let $\{Q_j\}_{j=1}^n$ denote the vertices of $\partial\Omega$, ordered counterclockwise, with interior angles $\omega_j$. For each $j$, let the open set $\Gamma_j$ be the edge connecting $Q_{j-1}$ and $Q_j$ $(Q_0=Q_n)$, so that
	\(
	\partial\Omega = \bigcup_{j=1}^n \bar\Gamma_j.
	\)
	For a vertex $Q_j$, the index $j$ is classified according to the boundary conditions on the adjacent edges $\Gamma_j$ and $\Gamma_{j+1}$ as the following indicator sets:
	\begin{itemize}
		\item $\mathcal D^2$: $\Gamma_j,\Gamma_{j+1}\subset \Gamma_D$;
		\setlength{\parskip}{-0.1em}
		\item $\mathcal N^2$: $\Gamma_j,\Gamma_{j+1}\subset \Gamma_N$;
		\setlength{\parskip}{-0.1em}
		\item $\mathcal M'$: $\Gamma_j\subset \Gamma_N,\;\Gamma_{j+1}\subset \Gamma_D$;
		\setlength{\parskip}{-0.1em}
		\item $\mathcal M''$: $\Gamma_j\subset \Gamma_D,\;\Gamma_{j+1}\subset \Gamma_N$.
	\end{itemize}
}


For the convenience of analysis, we impose the following assumptions throughout the paper:
\begin{assumption}\label{assum}
    (i) All interior angles of $\Omega$ are at most $\frac{\pi}{2}$, except for the largest interior angle $\omega:=\omega_\ell$ at the vertex $Q:=Q_\ell$, where $1\leq \ell \leq n$. In other words, we consider a polygon with at most a single vertex that gives rise to the Sapongyan paradox.
    (ii) Dirichlet boundary conditions are imposed on all boundary edges except for the two edges adjacent to $Q$. The boundary conditions on $\Gamma_\ell$ and $\Gamma_{\ell+1}$ as shown in \Cref{set-Omega} will be specified later.
    Equivalently,
    \begin{align}\label{set-boundary}
        \tilde\Gamma := \partial\Omega \setminus (\bar\Gamma_\ell \cup \bar\Gamma_{\ell+1}) \subset \bar \Gamma_D.
    \end{align}
\end{assumption}

Position the vertex $Q$ at the origin and align its adjacent edge $\Gamma_{\ell+1}$ with the positive horizontal axis. 
Let $(r, \theta)$ denote the polar coordinates. Then $\Omega$ lies within the cone bounded by $\theta = 0$ and $\theta = \omega$. Define the sector
\[
K_\omega^R := \left\{ (r\cos\theta, r\sin\theta) \in \Omega \;\big|\; (r,\theta) \in (0,R) \times (0,\omega) \right\},
\]
where $R$ is chosen so that $K_\omega^R \subset \Omega$. A schematic of the domain and the sector is shown in \Cref{set-Omega}.

\fig{[h]
	\centerline{
		\pict{
			[scale=2]
			\draw[black] (0,0) --++ (1,0) --++ (-1.3,1.1) --++ (-0.9,-1.8) --++ (1.4,0) -- (0,0);
			\fill[green!5!] (0,0) --++ (1,0) --++ (-1.3,1.1) --++(-0.9,-1.8) --++ (1.4,0);
			\node[green, scale=1.4] at (-0.82,-0.55) {$\Omega$};
			
			\draw[-stealth] (0.8,-0.2) -- (0.6,-0.2) -- (0.5,0);
			\node[right=0pt,scale=0.75] at (0.8,-0.2) {$\Gamma_{\ell+1}$};
			\draw[-stealth] (0.8,-0.45) -- (0.14,-0.45);
			\node[right=0pt,scale=0.75] at (0.8,-0.45) {$\Gamma_\ell$};
			
			\draw[blue, fill=blue!10!] (0,0) -- (0.6,0) arc (0:286:0.6) -- (0,0);
			\fill (0,0) circle (0.5pt);
			\node[below=6pt,right=0pt,scale=0.75] at (0,0) {$Q$};
			\draw[black,->] (0.1,0) arc (0:286:0.1);
			\node[above=7pt,right,scale=0.65] at (0,0) {$\omega$};
			\node [blue] at (-0.3,0.3) {$K_\omega^R$};
		}
	}
	\caption{$\Omega$ and $K_\omega^R$.}
	\label{set-Omega}
}

\subsubsection{\texorpdfstring{$L^2$}{L^2} basis function}
The Laplace operator $\Delta: V^2 \to L^2(\Omega)$ is injective, and its range $\mathcal S := \Delta(V^2)$ forms a closed subspace of $L^2(\Omega)$ with finite codimension \cite{grisvard1992singularities}. Let $\mathcal S^\perp$ denote the orthogonal complement of $\mathcal S$ in $L^2(\Omega)$. Then one has the orthogonal decomposition $\mathcal S \oplus \mathcal S^\perp = L^2(\Omega)$. 
$\mathcal S^\perp$ is not necessarily empty. For a polygon, if $\omega < \tfrac{\pi}{2}$, we have $\mathcal S = L^2(\Omega)$, and the solution of the Poisson problem belongs to $V^2$ whenever the source term is in $L^2(\Omega)$. However, if the index $\ell \in \mathcal D^2$ and $\omega > \pi$. it has been shown that $\dim(\mathcal S^\perp) = 1$ \cite{grisvard1992singularities, li2023ac}, i.e., $\mathcal S \subsetneq L^2(\Omega)$. In this case, in general $w \notin \mathcal S$, which implies that the solution $\bar{u}$ obtained from the mixed formulation \eqref{eqn-dec} does not lie in $V^2$. When the index $\ell$ belongs to other indicator sets defined in \Cref{dfn-idx}, situations where $\mathcal S \subsetneq L^2(\Omega)$ may also occur, depending on the largest interior angle $\omega$ and the type of indicator set.
Analogous to the case $\ell \in \mathcal D^2$, the orthogonal complement $\mathcal S^\perp$ for all other cases is finite-dimensional, with an explicitly determinable basis.


\begin{definition}\label{defn}
At vertex $Q$, we define the functions $\xi_{m}=\xi_{m}(r,\theta;\tau,R) \in L^2(\Omega)$ by
\begin{align}\label{def-xi}
    \xi_{m} := \chi s_{m} + \zeta_{m},
\end{align}
where the cut-off function $\chi = \chi(r;\tau,R)\in C^\infty(\Omega)$ satisfies
\begin{align*}
    \chi(r;\tau,R) =
    \begin{cases}
        1, & 0 \leq r \leq \tau R,\\[6pt]
        0, & r \geq R,
    \end{cases}
\end{align*}
$s_{m} = s_{m}(r,\theta)\in L^2(\Omega)$ depends on indicator at $Q$
and the interior angle $\omega$,
\begin{align}\label{s-def}
    s_{m} :=
    \begin{cases}
        r^{-\tfrac{\pi}{\omega}}\sin\!\left(\tfrac{\pi}{\omega}\theta\right),
            & \ell \in\mathcal D^2,\; \omega\in(\pi,2\pi),\; m=1,\\[6pt]
        r^{-\tfrac{\pi}{\omega}}\cos\!\left(\tfrac{\pi}{\omega}\theta\right),
            & \ell \in\mathcal N^2,\; \omega\in(\pi,2\pi),\; m=1,\\[6pt]
        r^{-\tfrac{\pi}{2\omega}}\sin\!\left(\tfrac{\pi}{2\omega}\theta\right),
            & \ell \in\mathcal M',\; \omega\in\left(\tfrac{\pi}{2},\tfrac{3\pi}{2}\right],\; m=1,\\[6pt]
        r^{-\tfrac{(2m-1)\pi}{2\omega}}\sin\!\left(\tfrac{(2m-1)\pi}{2\omega}\theta\right),
            & \ell \in\mathcal M',\; \omega\in\left(\tfrac{3\pi}{2},2\pi\right),\; m=1,2,\\[6pt]
        r^{-\tfrac{\pi}{2\omega}}\cos\!\left(\tfrac{\pi}{2\omega}\theta\right),
            & \ell \in\mathcal M'',\; \omega\in\left(\tfrac{\pi}{2},\tfrac{3\pi}{2}\right],\; m=1,\\[6pt]
        r^{-\tfrac{(2m-1)\pi}{2\omega}}\cos\!\left(\tfrac{(2m-1)\pi}{2\omega}\theta\right),
            & \ell \in\mathcal M'',\; \omega\in\left(\tfrac{3\pi}{2},2\pi\right),\; m=1,2,
    \end{cases}
\end{align}
and $\zeta_{m}\in H^1(\Omega)$ is defined as the variational solution of the Poisson problem
\ali{\label{eqn-zeta}
	\ca{
		-\Delta\zeta_{m}=\Delta(\chi s_{m})&\text{ in }\Omega,\\
		\zeta_{m}=0&\text{ on }\Gamma_D,\\
		\partial_{\mathbf n}\zeta_{m}=0&\text{ on }\Gamma_N.
	}
}
\end{definition}

By the cut-off function $\chi$, it follows
\[
\chi s_{m}=\partial_{\mathbf n}(\chi s_{m})=0 \quad \text{on } \tilde\Gamma.
\]  
The definition of $s_{m}$ in \eqref{s-def} implies that for $j=\ell, \ell+1$,
\ali{\label{eqn-zeta+}
	\ca{
		\chi s_{m}|_{\Gamma_j}=0 & \text{ if } \Gamma_j \subset \Gamma_D,\\
		\partial_{\mathbf n} (\chi s_{m})|_{\Gamma_j}=0 & \text{ if } \Gamma_j \subset \Gamma_N.
	}
}
Furthermore, we observe that  
\[
\Delta(\chi s_{m})|_{K_{\omega}^{\tau R} \cup (\Omega \setminus K_{\omega}^{R})}=0, 
\quad 
\Delta(\chi s_{m})\in C^\infty(K_{\omega}^{R}\setminus K_{\omega}^{\tau R}),
\]  
which implies $\Delta(\chi s_{m})\in L^2(\Omega)$. By elliptic regularity for \eqref{eqn-zeta}, it holds $\zeta_{m} \in H^1(\Omega)$.  

We introduce the maximal extension of the Laplace operator $\Delta$ in $L^2(\Omega)$:
\begin{equation*}
    D(\Delta, L^2(\Omega)) = \{v\in L^2(\Omega); \Delta v \in L^2(\Omega)\}.
\end{equation*}
For a function $v$, let $\gamma_j v$ be the restriction of $v$ to $\Gamma_j$. 
\begin{lemma} \cite[Theorem 1.5.2]{grisvard1992singularities}
Let $\Omega$ be an open polygonal subset of $\mathbb R^2$. Then the  mapping
$$
v \mapsto \{\gamma_j v, \gamma_j (\partial v / \partial n)\},
$$
which is defined for $H^2(\Omega)$ has an unique continuous extension as an operator from $D(\Delta, L^2(\Omega))$ into $\widetilde{H}^{-{1}/{2}}(\Gamma_j) \times \widetilde{H}^{-{3}/{2}}(\Gamma_j)$.
\end{lemma}

Therefore, the function $\xi_{m}$ in \eqref{def-xi} has the following properties.
\begin{lemma}\label{xi-unique}
(i) $\xi_{m}$ depends on $r$, $\theta$, but are independent of $\tau$, $R$. That is, for any two distinct parameter pairs $(\tau_l, R_l)$, $l=1,2$, satisfying 
\(\min\{\tau_1 R_1, \tau_2 R_2\}>\delta>0\), it follows
\[
\xi_{m}(r,\theta;\tau_1,R_1) = \xi_{m}(r,\theta;\tau_2,R_2).
\]
(ii) The function $\xi_{m}  \in D(\Delta, L^2(\Omega))$ is uniquely defined and satisfies
\begin{align}\label{eqn-xi}
\begin{cases}
-\Delta \xi_{m} = 0 & \text{in } \Omega, \\
\xi_{m} = 0 & \text{on } \Gamma_D, \\
\partial_{\mathbf n} \xi_{m} = 0 & \text{on } \Gamma_N.
\end{cases}
\end{align}
\end{lemma}
\begin{proof}
The proof of (i) is analogous to Lemma 2.4 in \cite{li2023ac}, while that of (ii) follows from Lemma 2.3.6 in \cite{grisvard1992singularities}.
\end{proof}


Since $\xi_m(r,\theta;\tau,R)$ is independent of $\tau$ and $R$ (see \Cref{xi-unique}), we may simply write
$$\xi_m(r,\theta):=\xi_m(r,\theta;\tau,R).$$


\subsubsection{Orthogonal space of the Laplace operator image}

We explicitly characterize the orthogonal complement $\mathcal{S}^\perp$. We begin by identifying the possible components contained in it.

\lem{\label{xi perp S}
    The function $\xi_m$ defined in \eqref{def-xi} belongs to $\mathcal{S}^\perp$. In particular, for $\ell \in \mathcal{M}' \cup \mathcal{M}''$ and $\omega \in ({3\pi}/{2}, 2\pi)$, the functions $\xi_m$, $m = 1, 2$, are linearly independent.
}

\begin{proof}
For any $v \in \mathcal{S}$, there exists $z \in V^2$ such that $-\Delta z = v$. Taking the inner product of $v$ with $\xi_m$ gives
	\al{
		\int_\Omega v \xi_m \mathrm dx
		& = \int_\Omega -\Delta z \: \xi_m \mathrm dx = \int_\Omega -z \Delta\xi_m \mathrm dx + \int_{\partial\Omega} z \partial_{\mathbf n} \xi_m - \partial_{\mathbf n} z \: \xi_m \mathrm ds \\
		& = \int_\Omega -z \Delta\xi_m \mathrm dx + \int_{\Gamma_D} -\partial_{\mathbf n}z \: \xi_m \mathrm ds + \int_{\Gamma_N} z \partial_{\mathbf n}\xi_m \mathrm ds = 0,
	}
where we have used \eqref{eqn-xi}.
Therefore, $\xi_m \in \mathcal{S}^\perp$.

If $\ell \in \mathcal{M}' \cup \mathcal{M}''$ and $\omega \in ({3\pi}/{2}, 2\pi)$, suppose that for some constants $k_1, k_2$，
\begin{align*}
k_1 \xi_1 + k_2 \xi_2 = k_1 \zeta_1 + k_2 \zeta_2 + \chi (k_1 s_1 + k_2 s_2)=0, 
\end{align*}
which implies $\chi (k_1 s_1 + k_2 s_2) = - (k_1 \zeta_1 + k_2 \zeta_2) \in H^1(\Omega)$. Since $s_m \not\in H^1(\Omega)$, $m=1,2$, it holds
\begin{equation}\label{smnon}
    k_1 s_1 + k_2 s_2 =0.
\end{equation}
Multiplying \eqref{smnon} by $r^{-\frac{\pi}{2\omega}}$ yields $k_2 r^{-\frac{\pi}{2\omega}} s_2 = - k_1 r^{-\frac{\pi}{2\omega}} s_1 \in L^2(\Omega)$. However, since $
r^{-\frac{\pi}{2\omega}} s_2  \not \in L^2(\Omega)$, it must be that $k_2=0$, and consequently $k_1=0$.
Hence, $\xi_1$ and $\xi_2$ are linearly independent.
\end{proof}

Regarding the dimension of $\mathcal{S}^\perp$ in a general polygonal domain, we recall the following result.
\lem{\label{thm-main}\cite[Theorem 2.3.7]{grisvard1992singularities}
    Given a general polygonal domain, the dimension of $\mathcal{S}^\perp$ is equal to the cardinality of the set $\{\lambda_{j,m}:\ 0 < \lambda_{j,m} < 1,\ m \geq 1\}$, that is,
\begin{equation}\label{dimsp}
    \dim \mathcal{S}^\perp = \mathrm{card}\{\lambda_{j,m}:\ 0 < \lambda_{j,m} < 1,\ m \geq 1\}.
\end{equation}
Here, $\lambda_{j,m}$ $(j=1,\ldots,n;\ m \geq 1)$ denote the eigenvalues of the boundary value problem
\begin{equation}\label{eigen}
    -\partial_\theta^2 \phi = \lambda^2 \phi, \qquad \theta \in (0,\omega_j), \quad \phi \in D_j,
\end{equation}
where the admissible sets $D_j$ are defined as
\[
D_j =
\begin{cases}
\{\phi \in H^2([0,\omega_j]) : \phi(0)=0,\ \phi(\omega_j)=0\}, & j \in \mathcal D^2, \\[1ex]
\{\phi \in H^2([0,\omega_j]) : \phi'(0)=0,\ \phi'(\omega_j)=0\}, & j \in \mathcal N^2, \\[1ex]
\{\phi \in H^2([0,\omega_j]) : \phi'(0)=0,\ \phi(\omega_j)=0\}, & j \in \mathcal M', \\[1ex]
\{\phi \in H^2([0,\omega_j]) : \phi(0)=0,\ \phi'(\omega_j)=0\}, & j \in \mathcal M''.
\end{cases}
\]
}

For a general polygonal domain, it can be readily verified that, for a fixed $j$, the solution of the eigenvalue problem \eqref{eigen} is given by
\[
\lambda_{j,m} =
\begin{cases}
m {\pi}/{\omega_j}, & j \in \mathcal D^2, \\[1ex]
(m-1) {\pi}/{\omega_j}, & j \in \mathcal N^2, \\[1ex]
\left(m-\tfrac{1}{2}\right) {\pi}/{\omega_j}, & j \in \mathcal M' \cup \mathcal M'',
\end{cases}
\qquad m=1,2,\ldots
\]
Then, the dimension of $\mathcal{S}^\perp$ in \eqref{dimsp} reduces to
\begin{align}\label{dimsgen}
    d_\perp := \dim \mathcal{S}^\perp=& \ \mathrm{card}\bbr{j\in\mathcal D^2\cup\mathcal N^2:\ \omega_j\in(\pi, 2\pi)}+\mathrm{card}\bbr{j\in\mathcal M'\cup\mathcal M'':\ \omega_j\in(\frac\pi2,\frac{3\pi}2]}\\
	&+2\mathrm{card}\bbr{j\in\mathcal M'\cup\mathcal M'':\ \omega_j\in(\frac{3\pi}2,2\pi)}, \nonumber
\end{align}
which shows that $\mathcal{S}^\perp$ is finite-dimensional, with its dimension determined by the interior angles $\omega_j$ and the types of boundary conditions prescribed on the polygonal domain.

Upon revisiting our earlier assumptions, we obtain to the following result.
\begin{corollary}\label{lem-perp}
    Under \Cref{assum}, the dimension of the orthogonal space $\mathcal{S}^\perp$ holds
    \begin{equation}\label{dimsperp}
        d_\perp = 
        \ca{
            1, \quad \ell\in \mathcal D^2\cup \mathcal N^2, \omega\in (\pi, 2\pi), \text{ or } \ell\in \mathcal M'\cup\mathcal M'', \omega\in (\frac12\pi, \frac32\pi], \\
            2, \quad \ell\in \mathcal M'\cup \mathcal M'', \omega\in (\frac32\pi, 2\pi).
        }
    \end{equation}
    Furthermore, $\mathcal{S}^\perp$ can be expressed as 
    \begin{equation}\label{sperp}
        \mathcal S^\perp = \mathop{span}_m\{\xi_m\},
    \end{equation}
    where $\xi_m$, $m=1, \dots, d_\perp$ is given in \eqref{def-xi}.
\end{corollary}
\begin{proof}
    By the definition of $\xi_m$ in \eqref{def-xi} and \Cref{thm-main}, it follows  $d_\perp = \mathop {card}\limits_m \{\xi_m\}$. 
    Additionally, \Cref{xi perp S} implies $\xi_m \in \mathcal{S}^\perp$, and hence \eqref{sperp} holds.
\end{proof}

    




\subsection{Modified mixed formulation and its well-posedness}
By \Cref{lem-perp}, any function $w\in L^2(\Omega)$ can be decoupled as
\ali{\label{w-dec}
	w=w_{\mathcal S}+w_{\mathcal S^\perp},
}
where $w_{\mathcal S}\in \mathcal S$ and $w_{\mathcal S^\perp}\in \mathcal S^\perp$, satisfying
\al{
	& w_{\mathcal S^\perp}=\sum_{m=1}^{d_\perp} c_m \xi_m, \text{ and } (w_{\mathcal S}, w_{\mathcal S^\perp}) =0.
}
Taking the inner product of both sides of \eqref{w-dec} with $\xi_{m}$ yields a linear system for the coefficients: 
\begin{equation}\label{linsys}
    \Xi C = W,
\end{equation}
where $\Xi = [(\xi_m, \xi_{m'})]_{m, m'=1}^{d_\perp} \in \mathbb{R}^{d_\perp \times d_\perp}$ is the associated Gram matrix, $C=[c_m]_{m=1}^{d_\perp} \in \mathbb{R}^{d_\perp}$ is the coefficient vector, and $W =[(w, \xi_m)]_{m=1}^{d_\perp}  \in \mathbb{R}^{d_\perp}$ is the right-hand side vector.
\begin{lemma}
    The linear system \eqref{linsys} admits a unique solution.
\end{lemma}
\begin{proof}
If $d_\perp = 1$, the coefficient is given explicitly by
\(
    c_1 = \frac{(w,\xi_1)}{\|\xi_1\|^2}.
\)
If $d_\perp = 2$, by \Cref{xi perp S}, the basis functions $\xi_1$ and $\xi_2$ are linearly independent. Therefore, the associated Gram matrix $\Xi$ satisfies 
\[
    |\Xi| = \|\xi_1\|^2 \|\xi_2\|^2 - (\xi_1,\xi_2)^2 > 0,
\]
which implies that $\Xi$ is nonsingular. Consequently, the coefficient vector $C$ is uniquely determined. 
\end{proof}

Then the naive mixed formulation \eqref{eqn-dec} can be refined into the modified mixed formulation
\ali{\label{eqn-decModified}
    \ca{
	-\Delta w=f \ &\text{in }\Omega,\\
	w=0 \ &\text{on }\Gamma_D,\\
	\partial_{\mathbf n}w=0 \ &\text{on }\Gamma_N,
    }\qquad 
    \ca{
	-\Delta\tilde u=w_{\mathcal S} \ &\text{in }\Omega,\\
	\tilde u=0 \ &\text{on }\Gamma_D,\\
	\partial_{\mathbf n}\tilde u=0 \ &\text{on }\Gamma_N.
    } 
}
The corresponding variational formulation seeks $\tilde {u}, w \in V^1$ such that
\seqn{\label{eqn-decModified-var}
	\nca{{}
		A(w,\phi)=(f,\phi), \quad\forall\phi\in V^1,\label{eqn-decModified-var-1}\\
		A(\tilde u,\psi)=(w_{\mathcal S},\psi), \quad\forall\psi\in V^1,\label{eqn-decModified-var-2}
	}
}

For the modified mix formulation \eqref{eqn-decModified}, we have the following results.
\thm{\label{equivalence}
For $f\in L^2(\Omega)$, the solution $\tilde{u}$ to the modified variational problem \eqref{eqn-decModified-var} is identical to the solution $u$ of the biharmonic problem \eqref{eqn-var}, i.e., $\tilde{u} = u \in V^2$.
}
\prf{
Since $w_\mathcal{S} \in \mathcal{S}$, it follows $\tilde u \in V^2$.
    For any $v\in V^2$, one has 
    \al{
        (\Delta \tilde u, \Delta v)_\Omega = (w - w_{\mathcal S^\perp}, \Delta v)_\Omega = - (w, \Delta v)_\Omega + (w_{\mathcal S^\perp}, \Delta v)_\Omega = (\nabla w, \nabla v)_\Omega - (w, \partial_{\mathbf n} v)_{\partial \Omega}.
    }
    Since $w\in V^1,\, v\in V^2$, it holds $(w, \partial_{\mathbf n} v)_{\partial \Omega} = (w, \partial_{\mathbf n} v)_{\Gamma_D} + (w, \partial_{\mathbf n} v)_{\Gamma_N} = 0.$
    Thus,
    \al{
        (-\Delta \tilde u, \Delta v)_\Omega = (\nabla w, \nabla v)_\Omega = (f, v)_\Omega.
    }
    Therefore, $\tilde u$ satisfies \eqref{eqn-var}.  By the uniqueness of the solution to \eqref{eqn-var} in $V^2$, it holds $\tilde u =u$.
}

In the following, we will not distinguish between $u$ and $\tilde u$, and will uniformly denote them as $u$.

\lem{
    The solution $(w, u)$ to the modified mixed formulation \eqref{eqn-decModified-var} satisfies
    \al{
        \|w\|_1 \lesssim \|f\|, \quad \|u\|_2 \lesssim \|f\|.
    }
}

\section{\texorpdfstring{$C^0$}{C^0} finite element method}\label{section-3}

Based on the modified mixed formulation, we propose a $C^0$ linear finite element method for the biharmonic problem \eqref{eqn-initial}, and derive the error estimates.

Denote by $\mathcal T_h$ the triangulation of $\Omega$. The $C^0$-Lagrange finite element space $V_h$ is given by
\begin{equation}\label{Vhspace}
    V_h:=\{v\in C^0(\Omega)\cap V^1:v|_K\subset P_1,\; \forall K\in\mathcal T_h\},
\end{equation}
where $P_1$ denotes the space of linear polynomials on each element $K$.

We propose a $C^0$ finite element algorithm for the biharmonic problem \eqref{eqn-initial} based on \eqref{eqn-decModified} in Algorithm \ref{algorithm-modified}. 

\renewcommand{\thealgocf}{3.1}

\begin{algorithm}[hbt!]
\normalsize
\medskip
\caption{$C^0$ finite element method}
\label{algorithm-modified}
\SetKwProg{Step}{Step 1: }{}{}
\Step{ Solve for $w_h \in V_h$ from
}{
    \begin{minipage}{0.8\linewidth}
        $$A(w_h, v_h) = (f, v_h), \quad \forall v_h \in V_h.$$
    \end{minipage}
}
\SetKwProg{Step}{Step 2: }{}{}
\Step{Solve for $\zeta_{m,h} \in V_h$,  $1\leq m \leq d_\perp$ from}
 {
        \begin{minipage}{0.8\linewidth}
            $$A(\zeta_{m,h}, v_h) = (\Delta (\chi s_m), v_h),\quad \forall v_h \in V_h,$$
            where $\chi$, $s_m$ are given in \Cref{defn};\\
            Compute $\xi_{m,h} = \zeta_{m,h} + \chi s_m$; 
        \end{minipage}
}
\SetKwProg{Step}{Step 3: }{}{}
\Step{
        Compute the coefficients
}{
\begin{minipage}{0.8\linewidth}
    \begin{equation}\label{coefmat}
        \Xi_h C_h = W_h,
    \end{equation}
    where $\Xi_h \in \mathbb{R}^{d_\perp \times d_\perp}$ is the coefficient matrix with the $mm'$th entry $(\xi_{m,h}, \xi_{m',h})$, $C_h=[c_{1,h}, \ldots, c_{d_\perp,h}]^\top \in \mathbb{R}^{d_\perp}$, and $W_h \in \mathbb{R}^{d_\perp}$ with the $m$th entry $(w_h, \xi_{m,h})$.    
\end{minipage}
}
\SetKwProg{Step}{Step 4: }{}{}
\Step{ Solve $u_h\in V_h$ from}
 {
        \begin{minipage}{0.8\linewidth}
             \begin{equation}\label{fem_u}
        A(u_h, v_h) = \left(w_h - \sum\limits_{m=1}^{d_\perp} c_{m,h} \xi_{m,h}, v_h \right), \quad \forall v_h \in V_h.
    \end{equation}
        \end{minipage}
}
\end{algorithm}

By the Lax-Milgram Theorem, the equations in Algorithm \ref{algorithm-modified} are well-posed. 
\begin{remark}\label{def-alpha}
    By \Cref{assum}, if $\ell \in \mathcal D^2 \cup \mathcal N^2$ for $\omega < \pi$ and $\ell \in \mathcal M' \cup \mathcal M''$ for $\omega \le \frac{\pi}{2}$, the solution satisfies $u \in V^2$. For these cases, Algorithm~\ref{algorithm-modified} reduces to the naive mixed finite element method.
\end{remark}
\begin{remark}
    Algorithm~\ref{algorithm-modified} extends naturally to the biharmonic equation with Neumann boundary conditions (\( \Gamma_N = \partial \Omega \)), which will be discussed in \Cref{section-4}, by substituting the finite element space \( V_h \) with \( \bar V_h \) defined in \eqref{barVhspace}.
\end{remark}

In every case, \(s_m\) in \eqref{s-def} has the form
\[
s_m(r,\theta)=r^{-\beta_{m}} \Phi_m(\theta),
\]
where \(\beta_{m}>0\) is the positive number appearing as the absolute value of the exponent of \(r\) in \eqref{s-def}, $\Phi_m$ denotes an appropriate trigonometric functions.  
By the regularity theory for the Poisson equation in polygonal domains (see \cite{grisvard1992singularities}), since \(w,\zeta_m\) solve the Poisson problems \eqref{eqn-decModified}, \eqref{eqn-zeta}, we have
\[
w,\ \zeta_m \in H^{1+\alpha}(\Omega),
\]
where $\alpha := \min_m \beta_m - \varepsilon$, 
and \(\varepsilon>0\) is chosen sufficiently small.
For simplicity of the error analysis below, we further assume that the solution exhibits a singular behavior characterized by $\alpha<1$.
Then, standard finite element error estimates yield the following result.
\lem{\label{PoissonError}
    Let $w$, $\zeta_m$, and $\xi_m$ be the solutions of \eqref{eqn-decModified}, \eqref{eqn-zeta}, and \eqref{def-xi}, respectively, and let $w_h$, $\zeta_{m,h}$, and $\xi_{m,h}$ be the corresponding finite element approximations in Algorithm \ref{algorithm-modified}. Then, 
    \al{
    	& \|w-w_h\|_1 \lesssim h^\alpha \|w\|_{1+\alpha}, \quad 
        \|w-w_h\| \lesssim h^{2\alpha} \|w\|_{1+\alpha}, \hspace{1cm}\\
    	& \|\xi_m-\xi_{m,h}\|_1 \lesssim h^\alpha \|\zeta_m\|_{1+\alpha}, \quad 
        \|\xi_m-\xi_{m,h}\| \lesssim h^{2\alpha} \|\zeta_m\|_{1+\alpha},
    } 
    where $w$ and $\zeta_m$ satisfy the following regularity estimates:
    \[
        \|w\|_{1+\alpha} \lesssim \|f\|, \quad
        \|\zeta_m\|_{1+\alpha} \lesssim \|\Delta (\chi s_m)\|.
    \]
}
\lem{\label{dualdiff}
Let $\gamma>0$ be a constant, and let $p, q$ be bounded scalars. If
\(
|p - p_h| \lesssim h^{\gamma}, \ |q - q_h| \lesssim h^{\gamma},
\)
then there exists $h_0>0$ such that for all $h<h_0$, the quantities $p_h$ and $q_h$ are uniformly bounded, and
\[
|pq - p_h q_h| \lesssim h^{\gamma}.
\]
}
\begin{proof}
By assumption, there exists $C>0$ such that 
\(|p-p_h|\le C h^{\gamma}\) and \(|q-q_h|\le C h^{\gamma}\) for all sufficiently small $h$.
Choose $h_0>0$ such that $C h_0^{\gamma} < 1$. Then for $h<h_0$,
\[
|p_h|\le |p| + |p-x_h| \le |p| + 1, 
\qquad
|q_h|\le |q| + |q-y_h| \le |q| + 1,
\]
so $p_h$ and $q_h$ are uniformly bounded.

Furthermore, it follows
\[
|pq - p_h q_h|
= |p(q - q_h) + (p - p_h)q_h|
\le |p|\,|q-q_h| + |p-p_h|\,|q_h|.
\]
Using the bounds above, we obtain
\[
|pq - p_h q_h|
\le |p|\, C h^{\gamma} + C h^{\gamma}(|q|+1)
\lesssim h^{\gamma}.
\]
\end{proof}
 
In terms of the finite element approximation $u_h$ in Algorithm \ref{algorithm-modified}, we have the following result.
\thm{\label{thmerr} Let $u$ be the solution to the modified mixed formulation \eqref{eqn-decModified-var}, and $u_h$ be  finite element approximation in Algorithm \ref{algorithm-modified}. Then,
\begin{align}\label{uleerr-}
    \|u-u_h\|_1 \lesssim h^{\min\{1,2\alpha\}}.
\end{align}
More specifically,
\begin{align}\label{uleerr}
    \|u-u_h\|_1 \lesssim  
        \ca{
            h^{2\alpha}, \quad & \text{ if } \ell\in \mathcal M'\cup \mathcal M'',\omega > \pi,\\
    		h, & \text{otherwise.} 
    	}
\end{align}
}
\begin{proof}
Taking $\psi=v_h$ in \eqref{eqn-decModified-var-2} and subtracting \eqref{fem_u} from \eqref{eqn-decModified-var-2} yield
\begin{equation}\label{galerkin}
    A(u-u_h, v_h) = \left((w - \sum_{m=1}^{d_\perp} c_m\xi_m) - (w_h - \sum_{m=1}^{d_\perp} c_{m,h}\xi_{m,h}),v_h \right).
\end{equation}

    Let $u_I \in V_h$ denote the Lagrange interpolation of $u$ onto $V_h$. Since $u \in V^2 \subset H^2(\Omega)$, the standard interpolation estimate on the triangulation $\mathcal{T}_h$ \cite{ciarlet2003singular} yields
    \al{
        \|u - u_I\|_1 
        \lesssim
        h \|u\|_2.
    }
    Denote $\varepsilon_h = u_I - u$ and $e_h=u_I-u_h$. 
	By taking $v_h = e_h$ in \eqref{galerkin}, it follows
    \al{
        |e_h|_1^2 
        & = (\nabla \varepsilon_h, \nabla e_h) + (w - w_h, e_h) - \sum_{m=1}^{d_\perp} (c_m - c_{m,h}) (\xi_m, e_h) - \sum_{m=1}^{d_\perp} c_{m,h} (\xi_m - \xi_{m,h}, e_h) \\
        & \leq |\varepsilon_h|_1|e_h|_1 + (\|w - w_h\| + \sum_{m=1}^{d_\perp} |c_m - c_{m,h}| \|\xi_m\| + \sum_{m=1}^{d_\perp} |c_{m,h}| \|\xi_m - \xi_{m,h}\|)\|e_h\| \\
        & \lesssim \left(|\varepsilon_h|_1 + \|w - w_h\| + \sum_{m=1}^{d_\perp} |c_m - c_{m,h}| \|\xi_m\| + \sum_{m=1}^{d_\perp} |c_{m,h}| \|\xi_m - \xi_{m,h}\|\right)\|e_h\|_1.
    }
    By Poincare's inequality \cite{meyers1978integral}, 
	\al{
		\|e_h\|_1^2 
        & \lesssim |e_h|_1^2 \lesssim (|\varepsilon_h|_1 + (\|w - w_h\| + \sum_{m=1}^{d_\perp} |c_m - c_{m,h}| \|\xi_m\| + \sum_{m=1}^{d_\perp} |c_{m,h}| \|\xi_m - \xi_{m,h}\|)\|e_h\|_1,
	}
    which implies 
    \begin{equation}\label{projerr}
        \|e_h\|_1 \lesssim |u - u_I|_1 + \|w - w_h\| + \sum_{m=1}^{d_\perp} |c_m - c_{m,h}| \|\xi_m\| + \sum_{m=1}^{d_\perp} |c_{m,h}| \|\xi_m - \xi_{m,h}\|.
    \end{equation}
	By \eqref{linsys} and \eqref{coefmat}, the coefficients can be expressed as 
	\al{
		c_m = \frac1{|\Xi|} \sum_{m'=1}^{d_\perp} \Xi_{(m',m)}(w, \xi_{m'}), 
		\qquad 
		c_{m,h} = \frac1{|\Xi_h|} \sum_{m'=1}^{d_\perp} \Xi_{h,(m',m)}(w_h, \xi_{m',h}).
	}
	The notations $\Xi_{(m',m)}$ and $\Xi_{h,(m',m)}$ denote the algebraic cofactors of $\Xi$ and $\Xi_h$, respectively. Specially, if $\dim \mathcal S^\perp = 1$, it holds $\Xi_{(1,1)} = \Xi_{h,(1,1)} = 1$. 
    
    By \Cref{PoissonError}, 
    there exists $h_0>0$ such that for all $h<h_0$, it holds $\|\xi_{m,h}\|_{1} \leq \|\xi_m\|_{1}+1$ and $\|w_h\|_{1} \leq \|w\|_{1}+1$.  
	Note that for $m,m'\geq 1$,
    \begin{align}
        \left|(\xi_m, \xi_{m'}) - (\xi_{m,h}, \xi_{m',h})\right|
		& = \left|(\xi_m, \xi_{m'} - \xi_{m',h}) + (\xi_m - \xi_{m,h}, \xi_{m',h})\right| \nonumber\\
		& \leq \|\xi_{m'} - \xi_{m',h}\| \|\xi_m\| + \|\xi_m - \xi_{m,h}\| \|\xi_{m',h}\| \lesssim h^{2\alpha}, \label{xixierr}\\
		\left|(w, \xi_{m'}) - (w_h, \xi_{m',h})\right|
		& \leq \left|(w, \xi_{m'} - \xi_{m',h}) + (w - w_h, \xi_{m',h})\right| \nonumber\\
		& \lesssim \|\xi_{m'} - \xi_{m',h}\| \|w\| + \|w - w_h\| \|\xi_{m',h}\| \lesssim h^{2\alpha}. \label{wxierr}
    \end{align}
Since each quantity $\Xi_{(m',m)}(w,\xi_{m'})$ and $|\Xi|$ is a linear combination of products of the terms $(\xi_m,\xi_{m'})$, $(w,\xi_{m'})$, and constants, it follows from \Cref{dualdiff} with $\gamma = 2\alpha$, together with \eqref{xixierr} and \eqref{wxierr}, that
\[
\left|\Xi_{(m',m)}(w,\xi_{m'}) - \Xi_{h,(m',m)}(w_h,\xi_{m',h})\right|
\lesssim h^{2\alpha},
\qquad
\left||\Xi| - |\Xi_h|\right|
\lesssim h^{2\alpha}.
\]
Then it follows
    \begin{align*}
        |c_m-c_{m,h}| 
		& = \left|\frac1{|\Xi|} \sum_{m'=1}^{d_\perp} \Xi_{(m',m)} (w, \xi_{m'}) - \frac1{|\Xi_h|} \sum_{m'=1}^{d_\perp} \Xi_{h,(m',m)} (w_h, \xi_{m',h})\right| \\
		& = \left|\frac1{|\Xi|} \sum_{m'=1}^{d_\perp} \left(\Xi_{(m',m)} (w, \xi_{m'}) -\Xi_{h,(m',m)}(w_h, \xi_{m',h})\right) + \left(\frac1{|\Xi|} - \frac1{|\Xi_h|}\right) \sum_{m'=1}^{d_\perp} 
        \Xi_{h,(m',m)} (w_h, \xi_{m',h})\right| \\
		& \leq \frac1{|\Xi|} \sum_{m'=1}^{d_\perp} \left|\Xi_{(m',m)} (w, \xi_{m'}) - \Xi_{h,(m',m)}(w_h, \xi_{m',h}) \right| + \frac{||\Xi_h|-|\Xi||}{|\Xi\Xi_h|} \sum_{m'=1}^{d_\perp} \left|\Xi_{h,(m',m)} (w_h, \xi_{m',h})\right| \\
		& \lesssim h^{2\alpha}.
    \end{align*}
    By \Cref{dualdiff}, $|c_{m,h}|$ is uniformly bounded. Substituting the estimate of $|c_m - c_{m,h}|$ into \eqref{projerr} gives
	\al{
		\| u- u_h\|_1 
		& \lesssim \| u - u_I\|_1 + \| u_I - u_h\|_1 \\
		& \lesssim \|u - u_I\|_1 + \|w - w_h\| + \sum_{m=1}^{d_\perp} |c_m - c_{m,h}| \|\xi_m\| + \sum_{m=1}^{d_\perp} |c_{m,h}| \|\xi_m - \xi_{m,h}\| \\
		& \lesssim h^{\min(1, 2\alpha)},
	}
    which gives the estimate \eqref{uleerr} as  $\alpha < \frac12$ when $\ell\in \mathcal M'\cup \mathcal M'',\ \omega > \pi$. 
\end{proof}
\begin{remark}
We note that the error estimate for high-order approximations using $P_k$ polynomials with $k\geq 2$ is not simply 
\[
\|u - u_h\|_1 \lesssim h^{\min\{k,\,2\alpha\}},
\]
as it can also be influenced by additional factors. A detailed investigation of these effects will be addressed in future research.
\end{remark}
\begin{remark}
In \Cref{thmerr}, the predicted convergence rate is \(h^{2\alpha}\) with \(2\alpha < 1\) when the largest interior angle exceeds \(\pi\) and is associated with mixed boundary conditions. However, extensive numerical experiments consistently demonstrate the optimal convergence rate of \(h\). This may be because the component of the solution influenced by the correction \(c_2 \xi_2\) in \(w_\mathcal{S}\) of the modified mixed formulation \eqref{eqn-decModified} is very small, although the correction remains necessary (see \Cref{ex4error} in \Cref{exa.6.1.315}). 
\end{remark}

\section{Neumann boundary conditions}\label{sec-neumann}\label{section-4}

Specifically, we consider the pure Neumann boundary case, that is, $\Gamma_D=\emptyset,\, \Gamma_N=\partial\Omega$. Then the biharmonic problem \eqref{eqn-initial} reduces to
\ali{\label{eqn-initial-neumann}
	\ca{
		\Delta^2u=f&\text{ in }\Omega,\\
		\partial_{\mathbf n}u=\Delta\partial_{\mathbf n} u=0&\text{ on }\partial\Omega.
	}
}
The variational formulation of \eqref{eqn-initial-neumann} is to find $u\in \bar{V}^2$ satisfying 
\ali{\label{eqn-var-neumann}
	a(u,v)=(f,v),\qquad\forall v\in \bar{V}^2,
}
where $f \in L^2(\Omega)$ satisfies the compatibility $\int_\Omega f \,\mathrm dx = 0$, and 
the Sobolev space $\bar V^2$ is defined by
$$\bar V^2=\{v\in H^2(\Omega): \partial_{\mathbf n} v|_{\partial\Omega} = 0, \ \int_\Omega v \,\mathrm dx = 0\}.$$

Similar to the general boundary conditions discussed in \Cref{sec2.1}, the problem \eqref{eqn-initial-neumann} can also be fully decomposed into the naive mixed formulation: two Poisson equations with Neumann boundary conditions
\ali{\label{eqn-dec-neumann}
	\ca{
		-\Delta w=f \ &\text{in }\Omega,\\
		\partial_{\mathbf n} w=0 \ &\text{on }\partial\Omega,
	}\qquad 
	\ca{
		-\Delta \bar u=w \ &\text{in }\Omega,\\
		\partial_{\mathbf n} u=0 \ &\text{on }\partial\Omega.
	}
}
The corresponding variational formulations are to find $\bar u, \ w\in \bar V^1$ such that
\ali{
    \ca{
		A(w,\phi)=(f,\phi), \quad\forall\phi\in \bar V^1,\\
		A(\bar u,\psi)=(w,\psi), \quad\forall\psi\in \bar V^1.
	}
}
Here, the Sobolev space $\bar V^1$ is defined by
$$
    \bar V^1 = \{v\in H^1(\Omega): \int_\Omega v \,\mathrm dx = 0\}.
$$
Note that $w\in \bar V^1$ implies $\int_\Omega w \, \mathrm dx = 0$, which guarantees the well-posedness of the second Poisson problem subject to the Neumann boundary condition.
The naive mixed formulation \eqref{eqn-dec-neumann} suffers from the same issues as the schemes corresponding to other boundary conditions discussed earlier. In the following part, we will introduce the corresponding modified mixed formulation.

\subsection{Modified mixed formulation} We follow the same assumption for the domain as described in \Cref{assum}(i). Specially, it holds $d_\perp=1$. 
The $L^2$ function $\xi$ (for simplicity, we omit the subscript $m$ in this section) is defined by
\ali{\label{def-xizeta-neumann}
    \xi = \chi s + \zeta, \qquad 
	s = r^{-\frac\pi\omega} \cos\frac\pi\omega\theta,\qquad 
	\ca{
		-\Delta\zeta=\Delta (\chi s)\quad&\text{in }\Omega,\\
		\partial_{\mathbf n}\zeta=0\quad&\text{on }\partial\Omega.
	}
}
Here, $\xi$ forms the basis function of the orthogonal space $\mathcal{S}^\perp$. Then, the modified mixed formuation is similarly given by
\ali{\label{eqn-decModified-neumann}
	\ca{
		-\Delta w=f \ &\text{in }\Omega,\\
		\partial_{\mathbf n}w=0 \ &\text{on }\partial\Omega,
	}\qquad 
	\ca{
		-\Delta\tilde u=w_{\mathcal S} \ &\text{in }\Omega,\\
		\partial_{\mathbf n}\tilde u=0 \ &\text{on }\partial\Omega,
	}
}
where 
\begin{equation}\label{wspdef}
    w_{\mathcal S} = w - w_{\mathcal S^\perp}, \qquad w_{\mathcal S^\perp} = \frac{(w, \xi)}{\|\xi\|^2} \xi.
\end{equation}
The corresponding variational formulation is to find $\tilde u, \ w\in \bar V^1$ satisfying
\ali{\label{eqn-decModified-var-neumann}
	\ca{
		A(w, \phi) = (f, \phi), \quad \forall \phi \in \bar V^1,\\
		A(\tilde u, \psi) = (w_{\mathcal S}, \psi), \quad \forall \psi \in \bar V^1.
	}
}
\begin{lemma}
The modified mixed formulation \eqref{eqn-decModified-neumann} or \eqref{eqn-decModified-var-neumann} admits a unique solution $w, \ \tilde{u} \in \bar V^1$.
\end{lemma}
\begin{proof}
Since $f$ satisfies the compatibility condition, the existence and uniqueness of $w$ are immediate. 
To establish the existence and uniqueness of $\tilde{u}$, it suffices to show that 
$w_{\mathcal S}$ in \eqref{wspdef} satisfies the compatibility condition
\(\int_\Omega w_{\mathcal S}\,\mathrm{d}x = \int_\Omega w dx - \frac{(w, \xi)}{\|\xi\|^2} \int_\Omega \xi dx =  0 \).
Since $w$ satisfies $\int_\Omega w\,\mathrm{d}x = 0$, it suffices to prove that
\(\int_\Omega \xi \,\mathrm{d}x = 0.\)

First, \eqref{def-xizeta-neumann} admits a unique solution $\zeta \in \bar V^1$, which follows from $\Delta(\chi s) \in L^2(\Omega)$ and 
\al{
	\int_\Omega\Delta(\chi s)\,\mathrm dx
	&=\int_\Omega(\partial_r^2+r^{-1}\partial_r+r^{-2}\partial_\theta^2)(\chi r^{-\frac\pi\omega}\cos\frac\pi\omega\theta)\,\mathrm dx\\
	&=\int_0^R r(\partial_r^2+r^{-1}\partial_r-\frac{\pi^2}{\omega^2})(\chi r^{-\frac\pi\omega})\,\mathrm dr\int_0^\omega\cos\frac\pi\omega\theta\,\mathrm d\theta = 0.
}
Then, it follows
\al{
	\int_\Omega \xi\,\mathrm dx
	& = \int_\Omega \zeta + \chi r^{-\frac\pi\omega} \cos \frac\pi\omega \theta \,\mathrm dx= \int_\Omega \zeta \,\mathrm dx + \int_0^R \chi  r^{1-\frac\pi\omega} \,\mathrm dr \int_0^\omega \cos \frac\pi\omega \theta \,\mathrm d\theta = 0.
}
\end{proof}
Similar to the general boundary cases, we have the following results.
\thm{The solution $\tilde{u}$ to the modified mixed formulation \eqref{eqn-decModified-neumann} is identical to the solution $u$ of the biharmonic problem \eqref{eqn-initial-neumann}, i.e., $\tilde u = u\in \bar V^2$.
Moreover, the solution $(w, u)$ to the modified mixed formulation \eqref{eqn-decModified-neumann} also satisfies
    \al{
        \|w\|_1 \lesssim \|f\|, \quad \|u\|_2 \lesssim \|f\|.
    }
}
Similar to the finite element space  $V_h$ in \eqref{Vhspace}, we introduce the $C^0$-Lagrange finite element space
\begin{equation}\label{barVhspace}
	\bar V_h:=\{v\in C^0(\Omega)\cap \bar V^1:v|_K\subset P_1,\; \forall K\in\mathcal T_h\},
\end{equation}
where $P_1$ denotes the space of linear polynomials on each element $K$. Then the $C^0$ finite element algorithm for the biharmonic problem \eqref{eqn-initial-neumann} based on the modified mixed formulation \eqref{eqn-decModified-neumann} can also be given by Algorithm \ref{algorithm-modified} with $V_h$ replaced by $\bar V_h$. 

For $\zeta, \ w$ in \eqref{def-xizeta-neumann} and \eqref{eqn-decModified-neumann}, it follows $w, \zeta \in  H^{1+\alpha}(\Omega)$ with $\alpha = \frac\pi\omega - \varepsilon$. 
\thm{\label{neumannrate}
Let $w_h, \xi_h, u_h$ be the finite element approximations of $C^0$ finite element algorithm for $w, \xi, u$ in the modified mixed formulation \eqref{eqn-decModified-neumann}.  Then it holds  
    \[
    	\|w - w_h\|_1 \lesssim h^\alpha, \hspace{2mm}
        \|w - w_h\| \lesssim h^{2\alpha}, \hspace{2mm}
    	\|\xi - \xi_h\|_1 \lesssim h^\alpha, \hspace{2mm}
        \|\xi - \xi_h\| \lesssim h^{2\alpha}, \hspace{2mm}
        \|u - u_h\|_1\lesssim h.
    \]
}

\section{Numerical Tests}\label{section-5}
In this section, we present numerical experiments on several model problems to verify the effectiveness of the proposed $C^0$ finite element algorithm, namely Algorithm \ref{algorithm-modified}. 
{\bf \noindent Domain shape}
For convenience of presentation, we consider the domains listed in \Cref{domains}, which are formed by removing parts of a square domain centered at the origin $Q$ with side length $4$.
\begin{figure}[H]
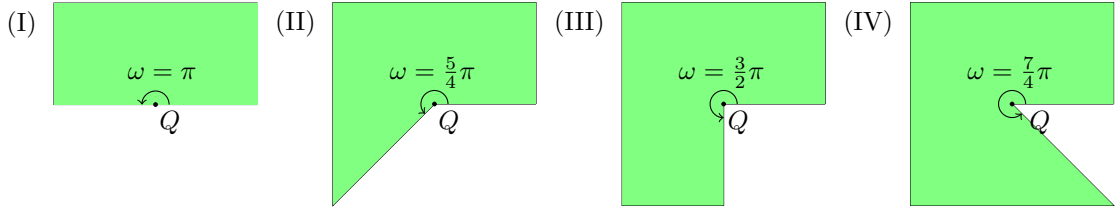

	\centering
	\pict{[scale=0.9]
		\draw[black] (1.5,0) -- (1.5,1.5) -- (-1.5,1.5) -- (-1.5,0) -- (1.5,0);
		\fill[green!50!] (1.5,0) -- (1.5,1.5) -- (-1.5,1.5) -- (-1.5,0) -- (1.5,0);
		\fill (0,0) circle (1pt);
		\node[above=-7pt,right=-2pt] at (0,0) {$Q$};
		\draw[->] (0.2,0) arc (0:180:0.2);
		\node[above=12pt,left=-19pt] at (0,0) {$\omega=\pi$};
		\node[above=-8pt,left=3pt] at (-1.5,1.5) {(I)};
		\fill [white!0!] (0,-1.5) -- (-.15,-1.5);
	}
	\pict{[scale=0.9]
		\draw[black] (1.5,0) -- (1.5,1.5) -- (-1.5,1.5) -- (-1.5,-1.5) -- (0,0) -- (1.5,0);
		\fill[green!50!] (1.5,0) -- (1.5,1.5) -- (-1.5,1.5) -- (-1.5,-1.5) -- (0,0) -- (1.5,0);
		\fill (0,0) circle (1pt);
		\node[above=-7pt,right=-2pt] at (0,0) {$Q$};
		\draw[->] (0.2,0) arc (0:225:0.2);
		\node[above=12pt,left=-19pt] at (0,0) {$\omega=\frac54\pi$};
		\node[above=-8pt,left=3pt] at (-1.5,1.5) {(II)};
	}
	\pict{[scale=0.9]
		\draw[black] (1.5,0) -- (1.5,1.5) -- (-1.5,1.5) -- (-1.5,-1.5) -- (0,-1.5) -- (0,0) -- (1.5,0);
		\fill[green!50!] (1.5,0) -- (1.5,1.5) -- (-1.5,1.5) -- (-1.5,-1.5) -- (0,-1.5) -- (0,0) -- (1.5,0);
		\fill (0,0) circle (1pt);
		\node[above=-7pt,right=-2pt] at (0,0) {$Q$};
		\draw[->] (0.2,0) arc (0:270:0.2);
		\node[above=12pt,left=-19pt] at (0,0) {$\omega=\frac32\pi$};
		\node[above=-8pt,left=3pt] at (-1.5,1.5) {(III)};
	}
	\pict{[scale=0.9]
		\draw[black] (1.5,0) -- (1.5,1.5) -- (-1.5,1.5) -- (-1.5,-1.5) -- (1.5,-1.5) -- (0,0) -- (1.5,0);
		\fill[green!50!] (1.5,0) -- (1.5,1.5) -- (-1.5,1.5) -- (-1.5,-1.5) -- (1.5,-1.5) -- (0,0) -- (1.5,0);
		\fill (0,0) circle (1pt);
		\node[above=-7pt,right=3pt] at (0,0) {$Q$};
		\draw[->] (0.2,0) arc (0:315:0.2);
		\node[above=12pt,left=-19pt] at (0,0) {$\omega=\frac74\pi$};
		\node[above=-8pt,left=3pt] at (-1.5,1.5) {(IV)};
	}
	\caption{Four different polygonal domains used in the numerical experiments. The re-entrant corner at point $Q$ has interior angle $\omega$. 
		I: $\omega=\pi$, II: $\omega=5\pi/4$, III: $\omega=3\pi/2$, IV: $\omega=7\pi/4$. }
	\label{domains}
\end{figure}

{\bf \noindent Boundary conditions}
Boundary conditions are categorized into the following five types:
\al{
	B_1:\, \partial\Omega = \bar \Gamma_D; \quad 
	B_2:\, \tilde \Gamma \subset \bar \Gamma_D, \ \ell\in \mathcal N^2; \quad 
	B_3:\, \tilde \Gamma \subset \bar \Gamma_D, \ \ell\in\mathcal M'; \quad 
	B_4:\, \tilde \Gamma \subset \bar \Gamma_D, \ \ell\in\mathcal M''; \quad 
	B_5:\, \partial\Omega = \bar \Gamma_N,
}
where $\tilde \Gamma$ is defined in \eqref{set-boundary}.

{\bf \noindent Cut-off function} Following \cite{li2023ac}, we consider the cut-off function
\al{
	\chi(r;\tau,R)=
	\ca{
		0,\quad r>R,\\
		1,\quad r<\tau R,\\
		-\frac3{16}\sbr{\frac{2r}{R(1-\tau)}-\frac{1+\tau}{1-\tau}}^5 + \frac58\sbr{\frac{2r}{R(1-\tau)}-\frac{1+\tau}{1-\tau}}^3 - \frac{15}{16}\sbr{\frac{2r}{R(1-\tau)}-\frac{1+\tau}{1-\tau}}+\frac12.
	}
}
For domains in \Cref{domains}, we take $R=1.8, \ \tau=0.125$.

Since analytical solutions to these biharmonic problems are unavailable, we employ the $C^0$ interior penalty discontinuous Galerkin ($C^0$-IPDG) method \cite{brenner2011c, cao2025posteriori}, which has been shown to produce numerical solutions that converge to the true solution independently of the geometry of $\Omega$ and the imposed boundary conditions.
In our computations, we use $P_2$ Lagrange elements, and the penalty parameter is chosen according to the specific problem under consideration.
After $7$ uniform mesh refinements, the resulting numerical solution, denoted by $u_R$, is taken as the reference solution.

In addition, we denote by $u_h^N$ the finite element solution based on the naive mixed formulation and by $u_h^M$ the finite element solution from Algorithm \ref{algorithm-modified}. To compute the convergence rate, we employ the Cauchy numerical convergence rate $\mathcal R(j)$ for a generic finite element solution $v_h$, defined as
\al{
	\mathcal R (j) = \log_2 \frac{|v_j-v_{j-1}|_1}{|v_{j+1}-v_j|_1}, \quad j \geq 1.
}
Here, $v_j$ denotes the finite element solution $v_h$ on the mesh $\mathcal T_h$ obtained after $j$ uniform refinements of the initial triangulation. 


\subsection{Mixed boundary conditions with \texorpdfstring{$\Gamma_D\neq\emptyset$}{}}

For all numerical tests in this subsection, unless otherwise specified, we set $f\equiv 1$ in $\Omega$.

%
%
%
%
%
%
%
%

\exa{\label{exa.6.1.180}
	Consider domain I, $\Omega=[-2,2]\times [0,2]$, with boundary types $B_3$ and $B_4$. 
	The corresponding triangulation and mesh refinements are shown in \Cref{ex1domain}.
	\fig{[h]
		\centering
		\subfloat [Initial mesh] {
			\includegraphics[width=1.1 in]{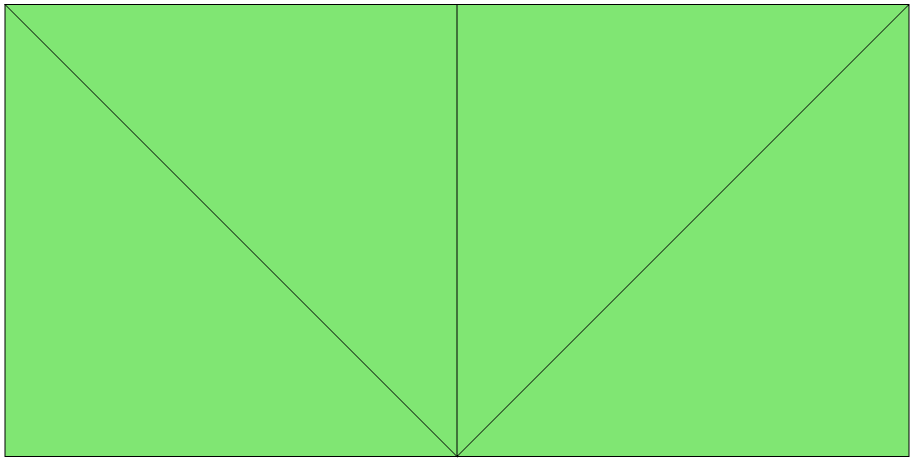}
		}\hspace{4em}
		\subfloat [$1$ refinement] {
			\includegraphics[width=1.1 in]{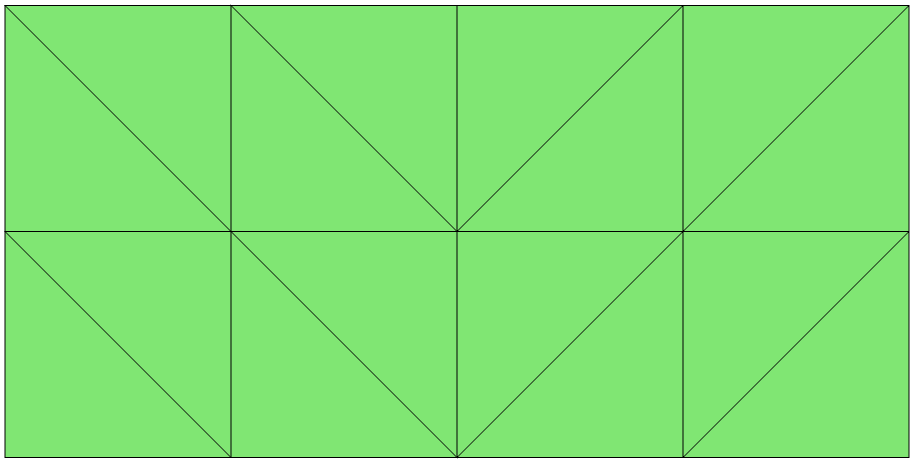}
		}
		\caption{\Cref{exa.6.1.180}: Initial mesh and mesh after $1$ refinement.}\label{ex1domain}
	}
	
	Consider the re-entrant corner $Q=Q_\ell$ with $\ell \in \mathcal M'\cup \mathcal M''$, having an interior angle $\omega = \pi$, located at the intersection of edges with Neumann and Dirichlet boundary conditions, respectively.
	In this case, $d_\perp = 1$. We compute the finite element solution $u_h^N$ and $u_h^M$, with the reference solution obtained using the $C^0$-IPDG method with penalty parameter $\sigma=50$. 
	The numerical solutions $u_h^N$ and $u_h^M$, computed on a mesh after eight uniform refinements, along with their differences from the reference solution, are presented in \Cref{ex1B34}. 
	The $L^\infty$ norm of the differences $|u_R-u_h^N|$, $|u_R-u_h^M|$ after $3$ to $6$ refinements are presented in \Cref{tab.180.err}.
	From these results, we observe that the naive mixed finite element solution $u_h^N$ converges to an incorrect solution, whereas the solution obtained from Algorithm~\ref{algorithm-modified} converges to the true solution. 
	\fig{[h]
		\centering
		\subfloat [$B_3, u_h^N$] {\label{fig.180.13.1.u.nai}
			\includegraphics[width=1.45 in]{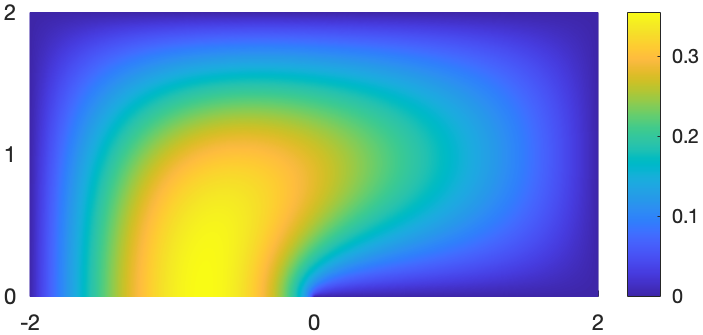}
		}
		\subfloat [$B_3, |u_R-u_h^N|$] {\label{fig.180.13.1.diff.nai}
			\includegraphics[width=1.45 in]{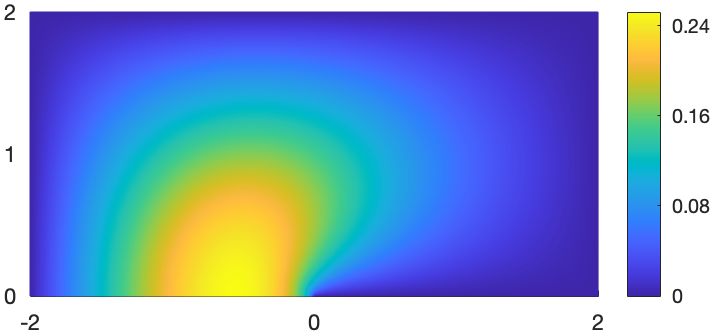}
		}
		\subfloat [$B_3, u_h^M$] {\label{fig.180.13.1.u.mod}
			\includegraphics[width=1.45 in]{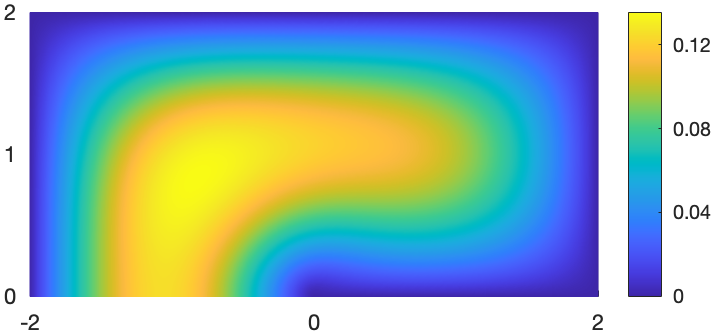}
		}
		\subfloat [$B_3, |u_R-u_h^M|$] {\label{fig.180.13.1.diff.mod}
			\includegraphics[width=1.45 in]{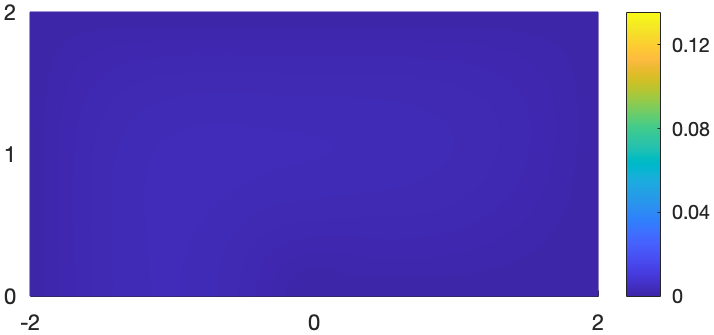}
		}\\ \vspace{0.2cm}
		\subfloat [$B_4, u_h^N$] {\label{fig.180.14.1.u.nai}
			\includegraphics[width=1.45 in]{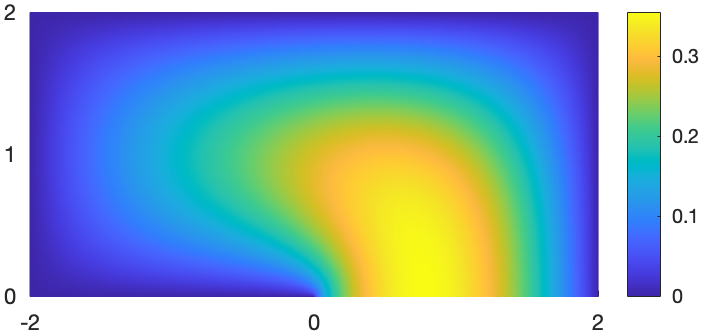}
		}
		\subfloat [$B_4, |u_R-u_h^N|$] {\label{fig.180.14.1.diff.nai}
			\includegraphics[width=1.45 in]{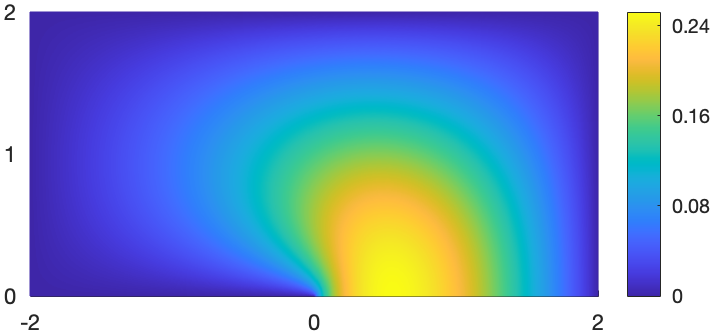}
		}
		\subfloat [$B_4, u_h^M$] {\label{fig.180.14.1.u.mod}
			\includegraphics[width=1.45 in]{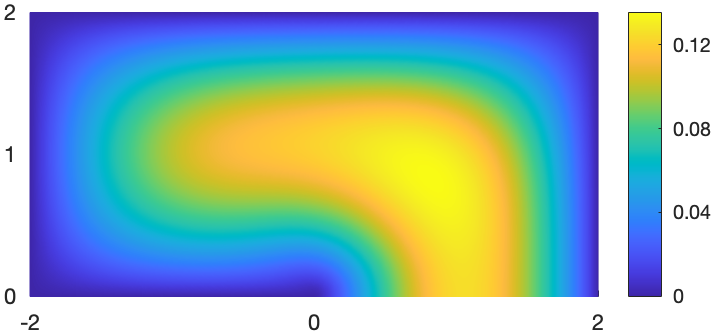}
		}
		\subfloat [$B_4, |u_R-u_h^M|$] {\label{fig.180.14.1.diff.mod}
			\includegraphics[width=1.45 in]{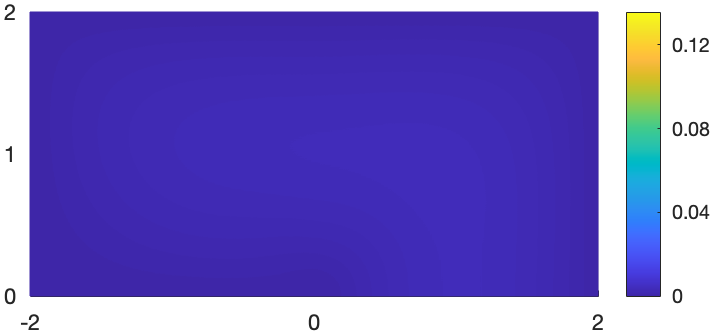}
		}
		\caption{\Cref{exa.6.1.180}: Domain I; Boundary type: $B_3, B_4$; $u_h^N, u_h^M$ and their differences with $u_R$.}\label{ex1B34}
		\label{fig.exa.180}
	}
	
	
	\begin{table}[h]
		\centering
		\caption{\Cref{exa.6.1.180}: Domain I, $L^\infty$ errors.}
		\label{tab.180.err}
		\begin{tabular}{c|c|cccc} 
			\hline
			& Boundary type & 3 & 4 & 5 & 6 \\ 
			\hline
			\multirow{2}{*}{$\|u_R-u_h^N\|_\infty$} 
			& $B_3$ & 2.1000e-01 & 2.3616e-01 & 2.4758e-01 & 2.5287e-01 \\
			& $B_4$ & 2.0998e-01 & 2.3614e-01 & 2.4756e-01 & 2.5285e-01 \\ 
			\hline
			\multirow{2}{*}{$\|u_R-u_h^M\|_\infty$} 
			& $B_3$ & 7.7457e-03 & 4.0231e-03 & 1.8178e-03 & 1.2360e-03 \\
			& $B_4$ & 7.7340e-03 & 4.0151e-03 & 1.8098e-03 & 1.2210e-03 \\
			\hline
		\end{tabular}
	\end{table}
	
	The convergence rates of $w_h$ and $u_h^M$ are shown in \Cref{tab.180.rate}. From these results, we observe that $u_h^M$ exhibits a convergence rate of order $h^1$, whereas the convergence rate of $w_h$ approaches $h^\alpha$ with $\alpha = 0.5$ under boundary types $B_3$ and $B_4$, which are consistent with the result in \Cref{PoissonError} and \Cref{thmerr}.

	\begin{table}[H]
		\centering
		\caption{\Cref{exa.6.1.180}: Domain I, convergence rates of $w_h,u_h^M$.}
		\label{tab.180.rate}
		\begin{tabular}{c|c|llllll}
			\hline
			& Boundary type & j=3 & j=4 & j=5 & j=6 & j=7 & j=8 \\ \hline
			\multirow{2}{*}{$\mathcal R$ for $u_h^M$ } 
			& $B_3$ & 0.82  & 0.91  & 0.97  & 0.99  & 0.99  & 1.00  \\ 
			& $B_4$ & 0.82  & 0.91  & 0.97  & 0.99  & 0.99  & 1.00  \\ \hline
			\multirow{2}{*}{$\mathcal R$ for $w_h$ } 
			& $B_3$ & 0.77  & 0.76  & 0.70  & 0.63  & 0.57  & 0.54  \\ 
			& $B_4$ & 0.77  & 0.76  & 0.70  & 0.63  & 0.57  & 0.54  \\ 
			\hline
		\end{tabular}
	\end{table}
	
}
%
%
%
%
%
%
%
%
\exa{\label{ex6.1.225}
	We repeat \Cref{exa.6.1.180} in domain II, $\Omega = [-2,2]^2\backslash\{(x,y)|y<0,y<x\}$, with boundary types $B_3$ and $B_4$. 
	The initial mesh and the mesh after one refinement are shown in \Cref{ex2domain}.
	\fig{[h]
		\centering
		\subfloat [Initial mesh] {
			\includegraphics[width=1.1 in]{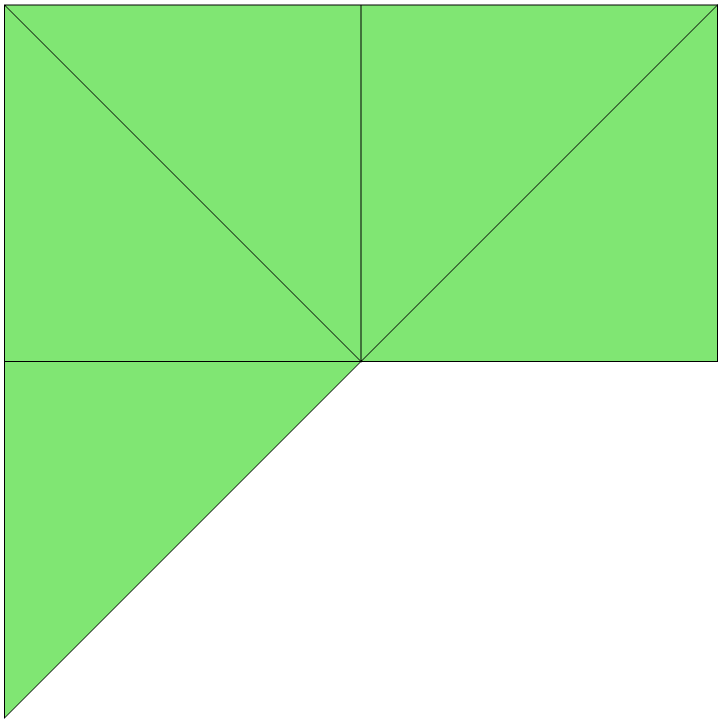}
		}\hspace{4em}
		\subfloat [$1$ refinement] {
			\includegraphics[width=1.1 in]{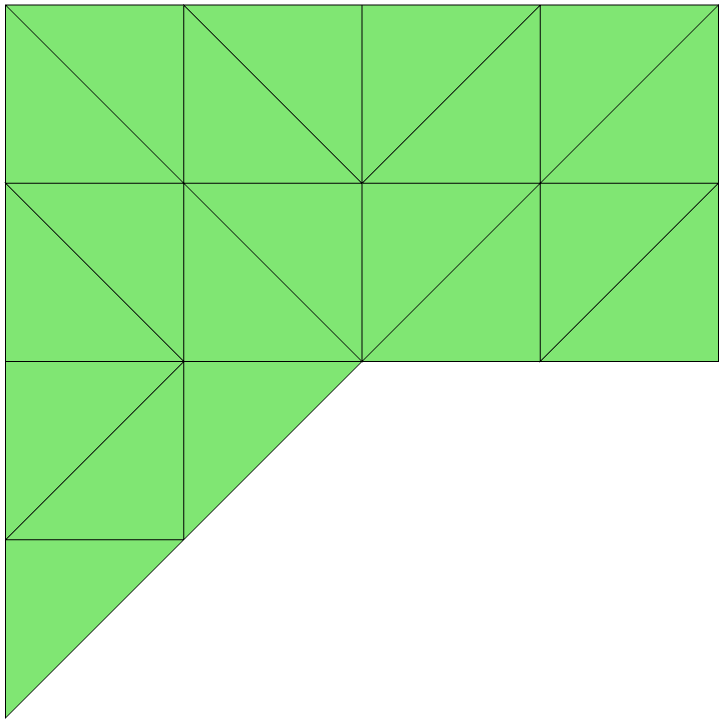}
		}
		\caption{\Cref{ex6.1.225}: Initial mesh and mesh after $1$ refinement.}\label{ex2domain}
	}
	
	In this case, the largest interior angle $\omega=\frac54\pi$, and $d_\perp = 1$. 
	The numerical solutions $u_h^N$ and $u_h^M$, computed on a mesh after eight uniform refinements, along with their differences from the reference solution, are presented in \Cref{fig.exa.225}. The $L^\infty$ norm of the differences $|u_R-u_h^N|$, $|u_R-u_h^M|$ after $3$ to $6$ refinements are presented in \Cref{ex2linf}.
	From these results, we further observe that the naive mixed finite element solution $u_h^N$ converges to an incorrect solution, whereas the solution obtained from Algorithm~\ref{algorithm-modified} converges to the true solution. 
	
	\fig{[h]
		\centering
		\subfloat [$B_3, u_h^N$] {\label{fig.225.13.1.u.nai}
			\includegraphics[width=1.45 in]{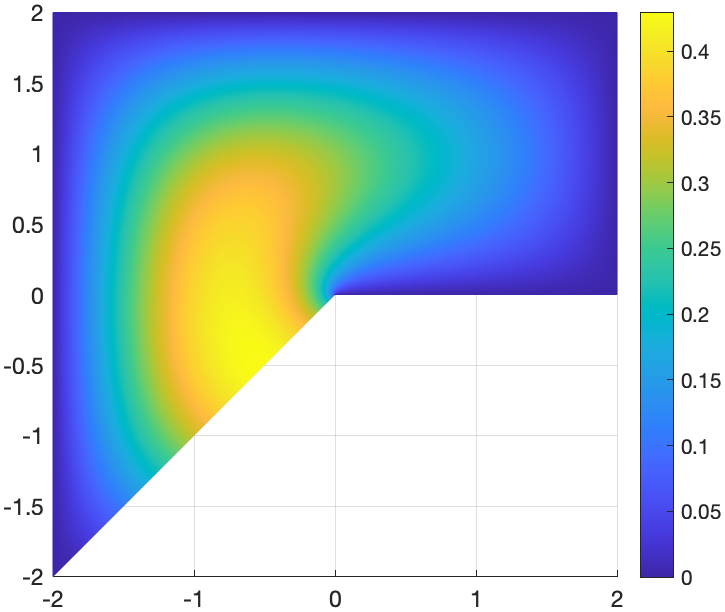}
		}
		\subfloat [$B_3, |u_R-u_h^N|$] {\label{fig.225.13.1.diff.nai}
			\includegraphics[width=1.45 in]{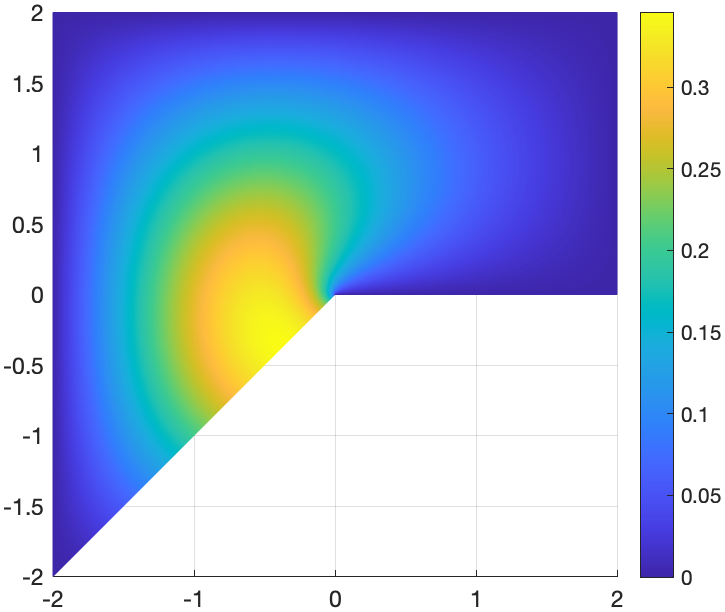}
		}
		\subfloat [$B_3, u_h^M$] {\label{fig.225.13.1.u.mod}
			\includegraphics[width=1.45 in]{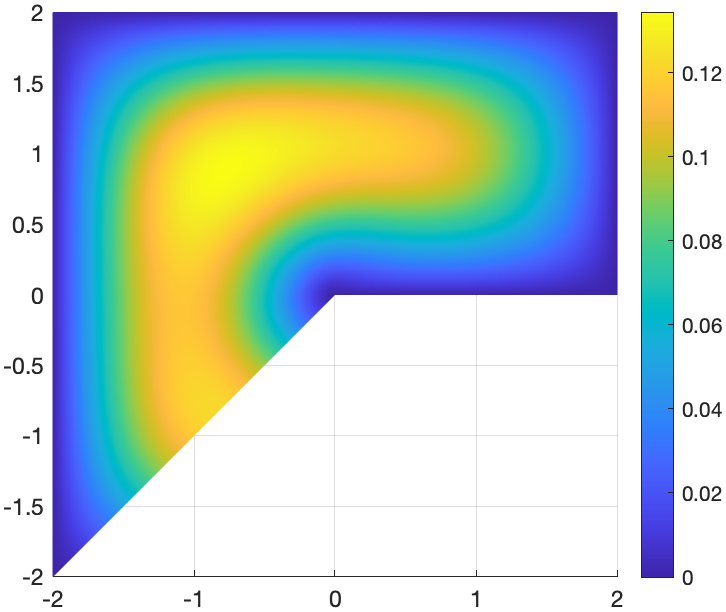}
		}
		\subfloat [$B_3, |u_R-u_h^M|$] {\label{fig.225.13.1.diff.mod}
			\includegraphics[width=1.45 in]{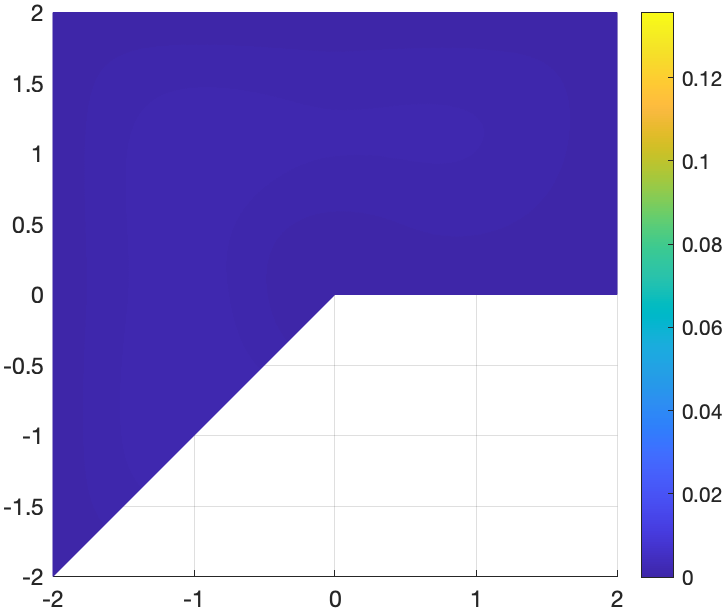}
		}\\ \vspace{0.2cm}
		\subfloat [$B_4, u_h^N$] {\label{fig.225.14.1.u.nai}
			\includegraphics[width=1.45 in]{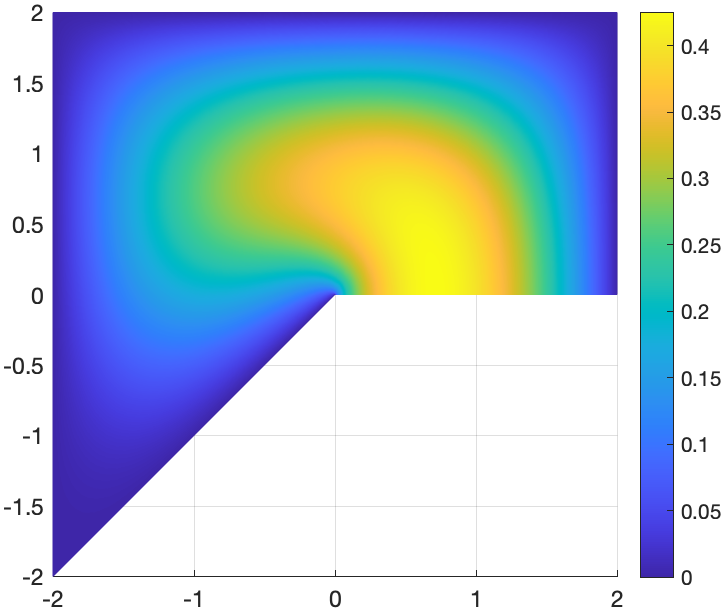}
		}
		\subfloat [$B_4, |u_R-u_h^N|$] {\label{fig.225.14.1.diff.nai}
			\includegraphics[width=1.45 in]{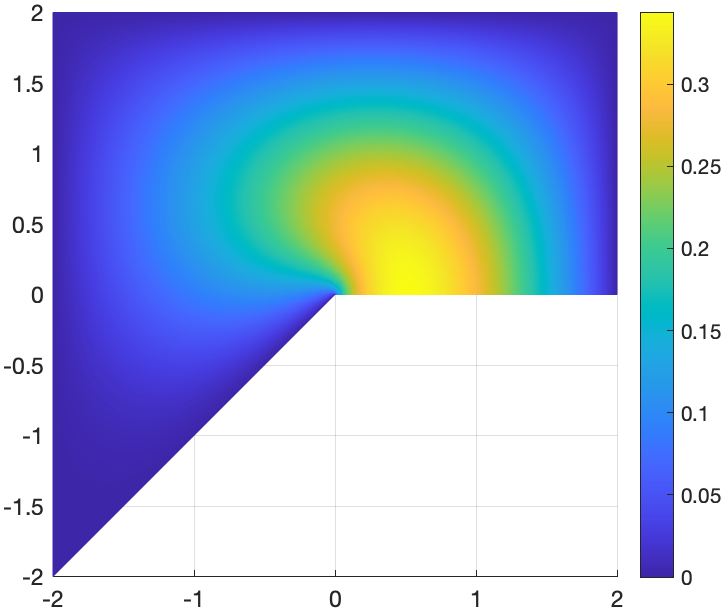}
		}
		\subfloat [$B_4, u_h^M$] {\label{fig.225.14.1.u.mod}
			\includegraphics[width=1.45 in]{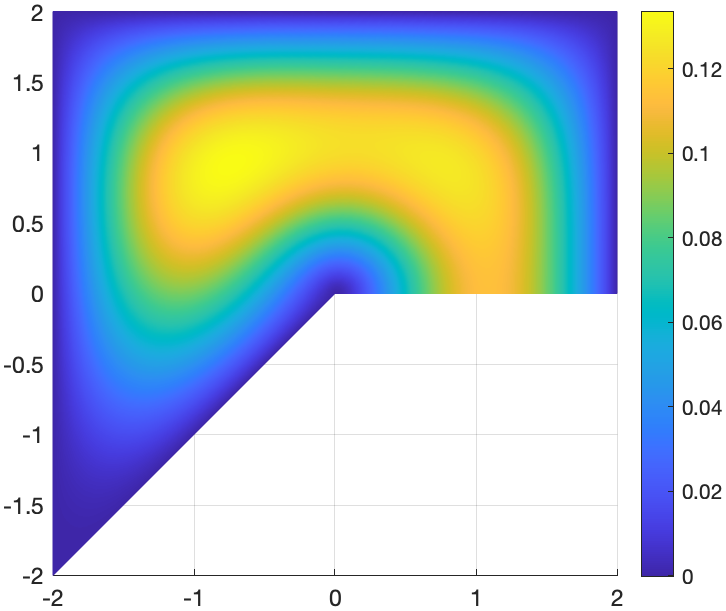}
		}
		\subfloat [$B_4, |u_R-u_h^M|$] {\label{fig.225.14.1.diff.mod}
			\includegraphics[width=1.45 in]{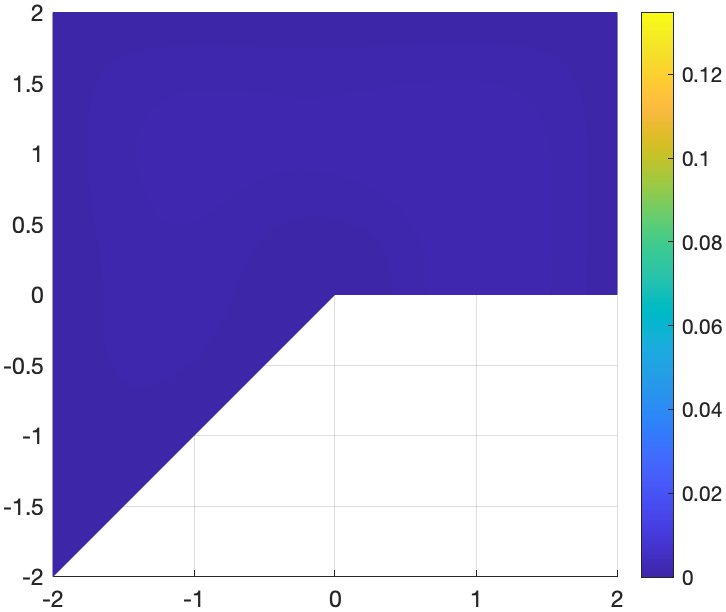}
		}
		\caption{\Cref{ex6.1.225}: Domain II; Boundary type: $B_3, B_4$; $u_h^N, u_h^M$ and their differences with $u_R$.}
		\label{fig.exa.225}
	}
	
	
	
	\begin{table}[h]
		\centering
		\caption{\Cref{ex6.1.225}: Domain II, $L^\infty$ errors.}\label{ex2linf}
		\begin{tabular}{c|c|cccc} 
			\hline
			& Boundary type & 3 & 4 & 5 & 6 \\ 
			\hline
			\multirow{2}{*}{$\|u_R-u_h^N\|_\infty$} 
			& $B_3$ & 2.6724e-01 & 3.0462e-01 & 3.2480e-01 & 3.3618e-01 \\
			& $B_4$ & 2.6306e-01 & 3.0116e-01 & 3.2253e-01 & 3.3407e-01 \\ 
			\hline
			\multirow{2}{*}{$\|u_R-u_h^M\|_\infty$} 
			& $B_3$ & 8.4341e-03 & 3.9003e-03 & 1.6030e-03 & 1.0728e-03 \\
			& $B_4$ & 8.1807e-03 & 3.8357e-03 & 1.5541e-03 & 1.0044e-03 \\
			\hline
		\end{tabular}
	\end{table}
	
	The convergence rates of $w_h$ and $u_h^M$ are shown in \Cref{d2rate}. From these results, we observe that $u_h^M$ exhibits a convergence rate of order $h^1$, whereas the convergence rate of $w_h$ approaches $h^\alpha$ with $\alpha = 0.4$ under boundary types $B_3$ and $B_4$, which are consistent with the result in \Cref{PoissonError} and \Cref{thmerr}.
	
	\begin{table}[h]
		\centering
		\caption{\Cref{ex6.1.225}: Domain II, convergence rates of $w_h,u_h^M$.}\label{d2rate}
		\begin{tabular}{c|c|llllll}
			\hline
			& Boundary type & j=3 & j=4 & j=5 & j=6 & j=7 & j=8 \\ \hline
			\multirow{2}{*}{$\mathcal R$ for $u_h^M$ }  & $B_3$ & 0.83  & 0.91  & 0.97  & 0.99  & 0.99  & 1.00  \\ 
			& $B_4$ & 0.83  & 0.91  & 0.97  & 0.99  & 0.99  & 1.00  \\ \hline
			\multirow{2}{*}{$\mathcal R$ for $w_h$ } & $B_3$ & 0.72  & 0.67  & 0.58  & 0.50  & 0.45  & 0.42  \\ 
			& $B_4$ & 0.73  & 0.66  & 0.57  & 0.49  & 0.45  & 0.42  \\ 
			\hline
		\end{tabular}
	\end{table}
	
}

%
%
%
%
%
%
%
%
\exa{\label{exa.6.1.270}
	Consider the $L$-shaped domain III, $\Omega=[-2,2]^2\backslash\{(x,y)|x>0,y<0\}$; the initial mesh and the mesh after $1$ refinement are  shown in \Cref{ex3mesh}.
	This domain contains an interior angle of $\frac32\pi$; we examine boundary types $B_1, B_2, B_3$, and $B_4$, with $d_\perp = 1$ for all  cases. 
	\fig{[hbt!]
		\centering
		\subfloat [Initial mesh] {
			\includegraphics[width=1.1 in]{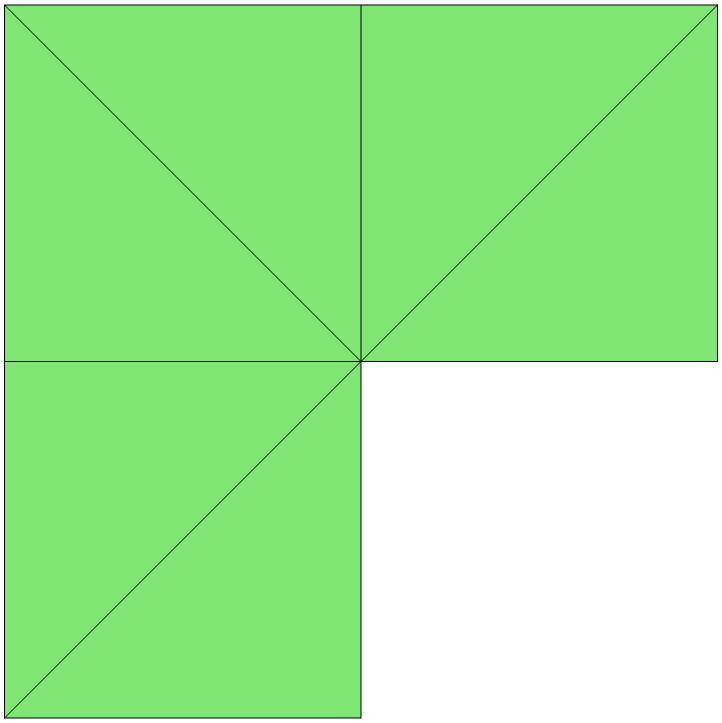}
		}\hspace{4em}
		\subfloat [$1$ refinement] {
			\includegraphics[width=1.1 in]{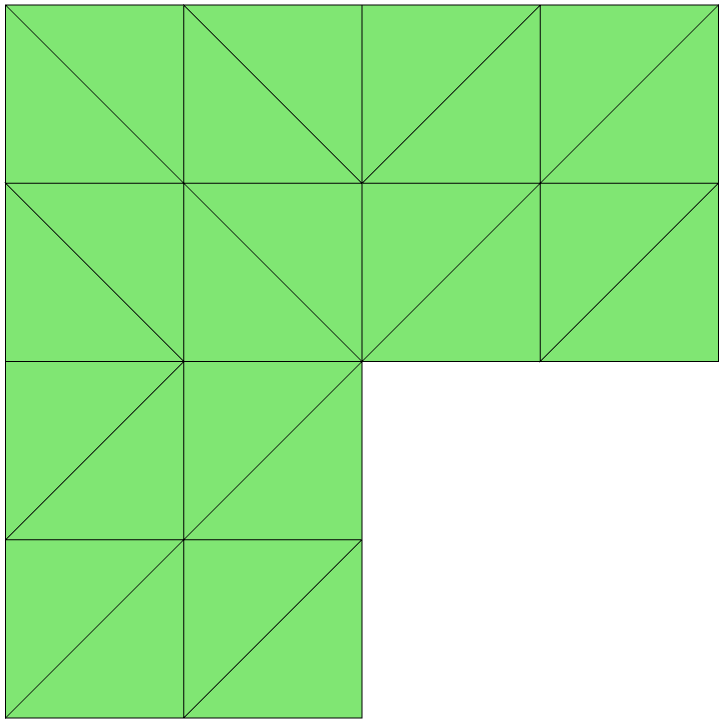}
		}
		\caption{\Cref{exa.6.1.270}: Initial mesh and mesh after $1$ refinement.}\label{ex3mesh}
	}
	For these four cases, we compute the finite element solution $u_h^N$ and $u_h^M$, with the reference solution $u_R$ obtained using the $C^0$-IPDG method with penalty parameter $\sigma=32,160,50,50$, respectively. 
	The numerical solutions $u_h^N$ and $u_h^M$, computed on a mesh after eight uniform refinements, along with their differences from the reference solution, are presented in \Cref{fig.exa.270.1} for boundary types $B_1$ and $B_2$, and \Cref{fig.exa.270.2} or boundary types $B_3$ and $B_4$.
	The $L^\infty$ norm of $|u_R-u_h^N|$, $|u_R-u_h^M|$ after 3 to 6 times refinements of the mesh are presented in \Cref{ex3err}. 
	From these results, we also observe that the naive mixed finite element solution $u_h^N$ converges to an incorrect solution in all cases except when the boundary type is $B_2$, whereas the solution obtained from Algorithm~\ref{algorithm-modified} converges to the true solution. 
	\fig{[h]
		\centering
		\subfloat [$B_1, u_h^N$] {\label{fig.270.11.1.u.nai}
			\includegraphics[width=1.45 in]{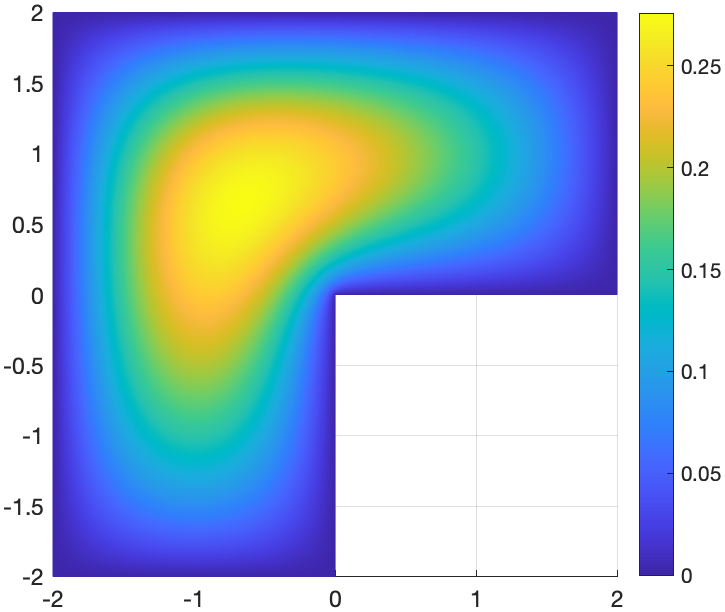}
		}
		\subfloat [$B_1, |u_R-u_h^N|$] {\label{fig.270.11.1.diff.nai}
			\includegraphics[width=1.45 in]{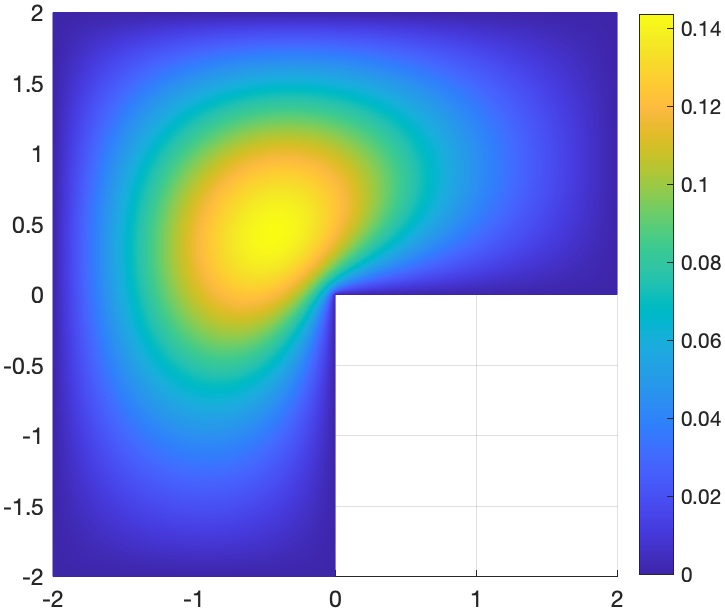}
		}
		\subfloat [$B_1, u_h^M$] {\label{fig.270.11.1.u.mod}
			\includegraphics[width=1.45 in]{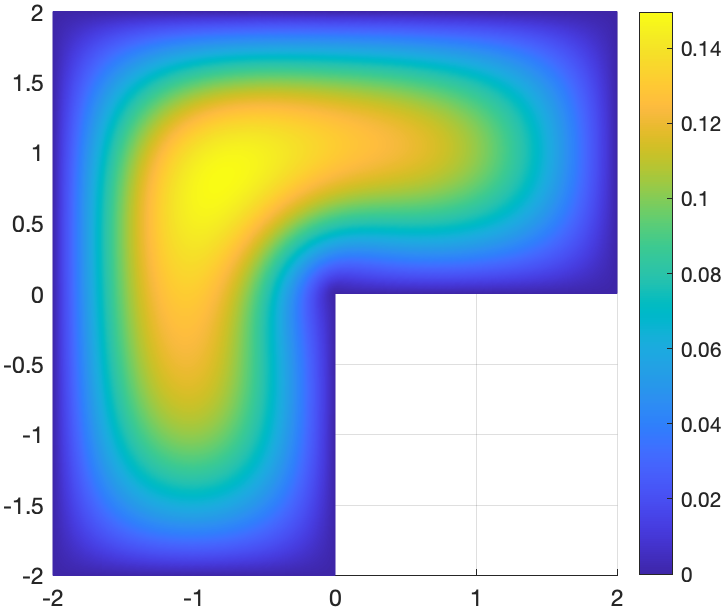}
		}
		\subfloat [$B_1, |u_R-u_h^M|$] {\label{fig.270.11.1.diff.mod}
			\includegraphics[width=1.45 in]{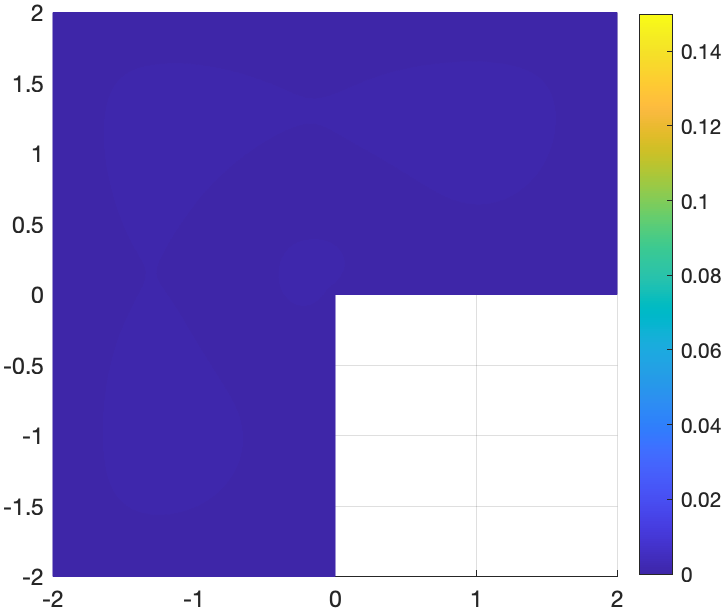}
		}\\ \vspace{0.1cm}
		\subfloat [$B_2, u_h^N$] {\label{fig.270.12.1.u.nai}
			\includegraphics[width=1.45 in]{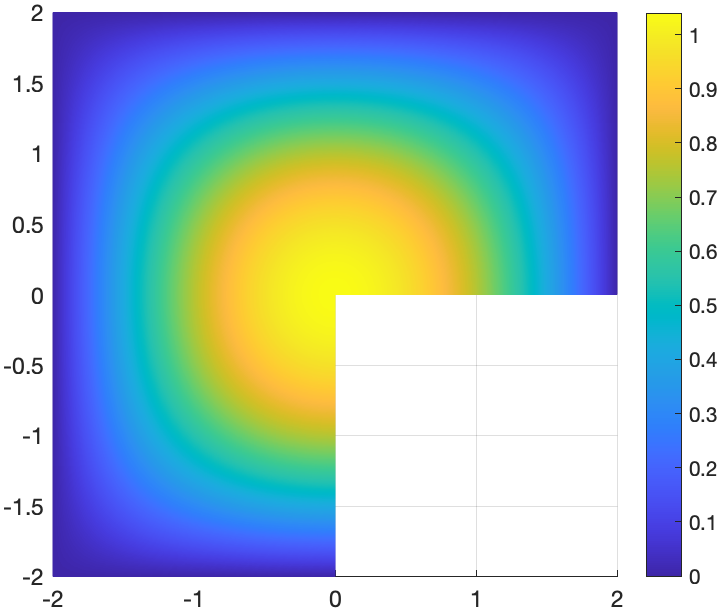}
		}
		\subfloat [$B_2, |u_R-u_h^N|$] {\label{fig.270.12.1.diff.nai}
			\includegraphics[width=1.43 in]{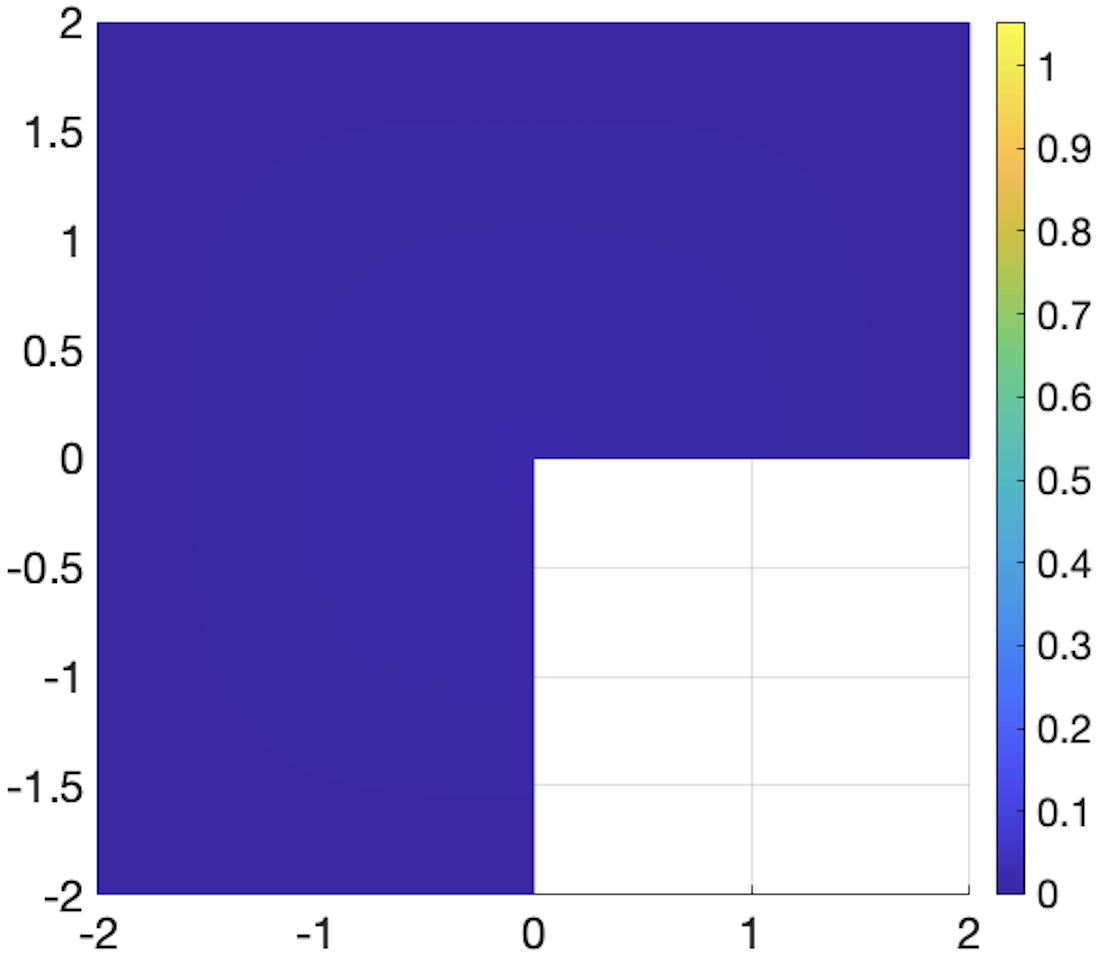}
		}
		\subfloat [$B_2, u_h^M$] {\label{fig.270.12.1.u.mod}
			\includegraphics[width=1.45 in]{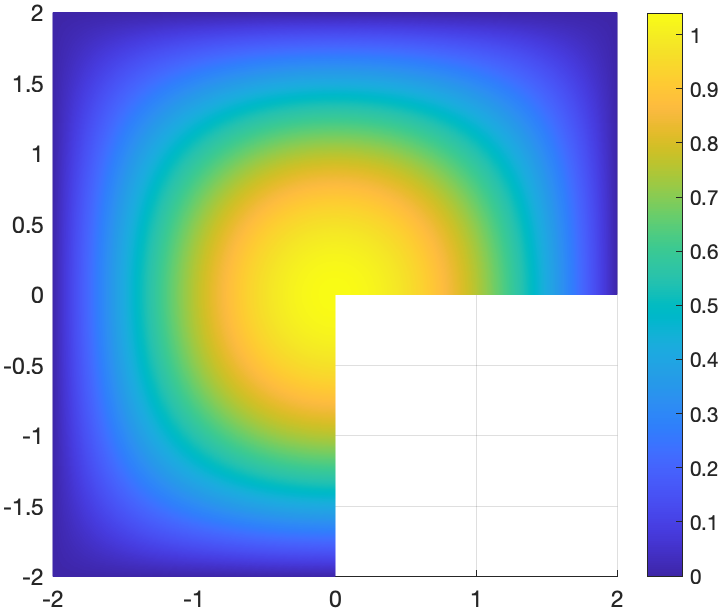}
		}
		\subfloat [$B_2, |u_R-u_h^M|$] {\label{fig.270.12.1.diff.mod}
			\includegraphics[width=1.45 in]{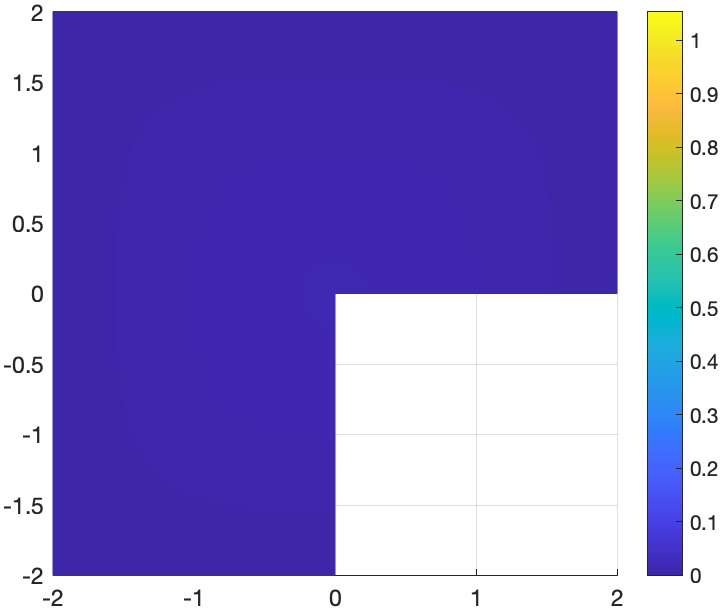}
		}
		\caption{\Cref{exa.6.1.270}: Domain III; Boundary type: $B_1, B_2$; $u_h^N, u_h^M$ and their differences with $u_R$.}
		\label{fig.exa.270.1}
	}
	
	\fig{[hbt!]
		\centering
		\subfloat [$B_3, u_h^N$] {\label{fig.270.13.1.u.nai}
			\includegraphics[width=1.45 in]{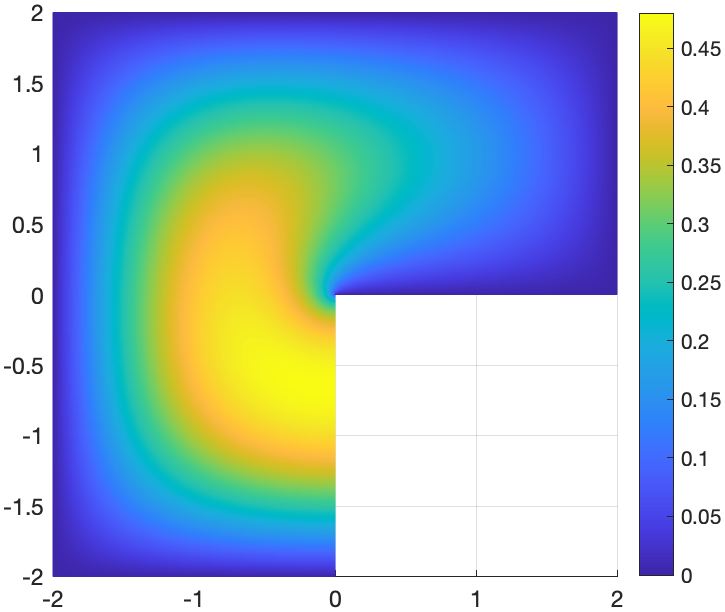}
		}
		\subfloat [$B_3, |u_R-u_h^N|$] {\label{fig.270.13.1.diff.nai}
			\includegraphics[width=1.45 in]{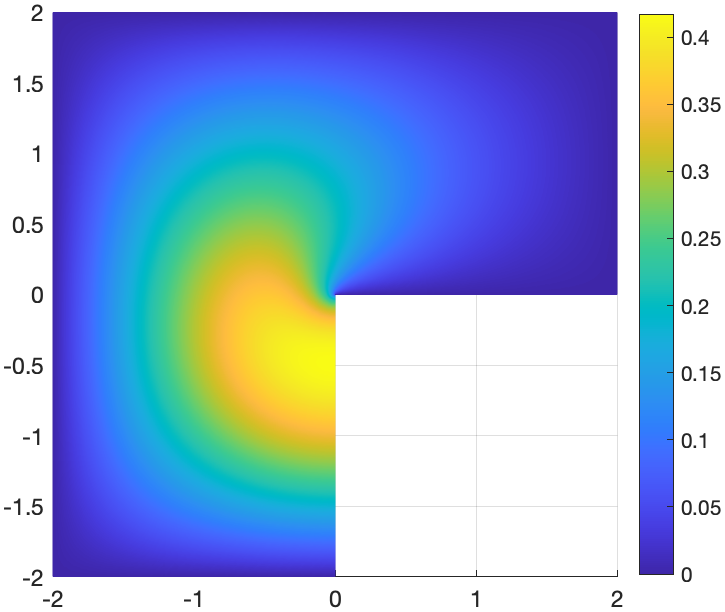}
		}
		\subfloat [$B_3, u_h^M$] {\label{fig.270.13.1.u.mod}
			\includegraphics[width=1.45 in]{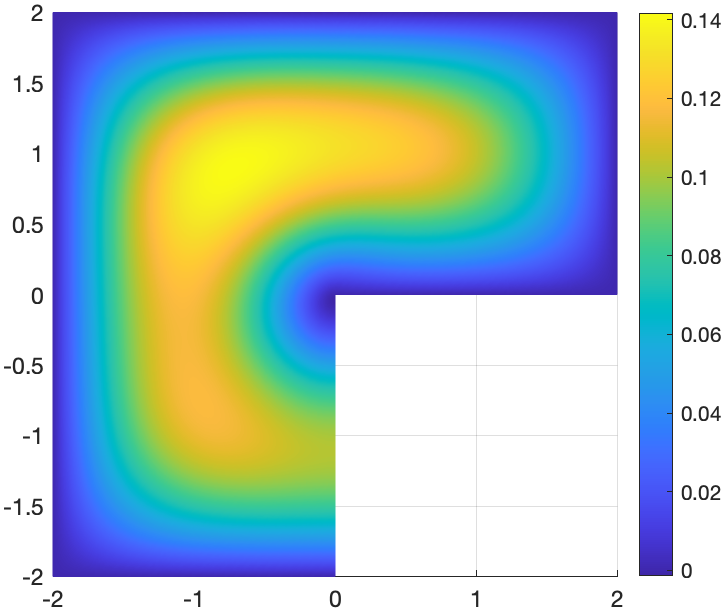}
		}
		\subfloat [$B_3, |u_R-u_h^M|$] {\label{fig.270.13.1.diff.mod}
			\includegraphics[width=1.45 in]{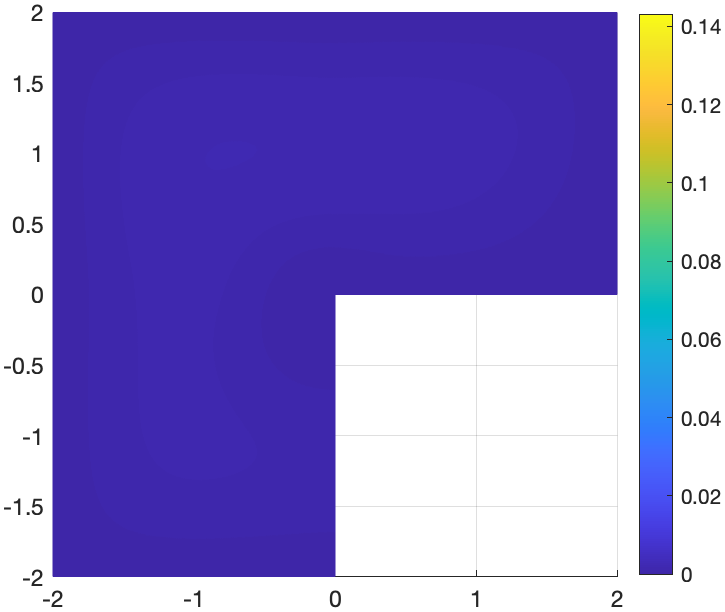}
		}\\ \vspace{0.1cm}
		\subfloat [$B_4, u_h^N$] {\label{fig.270.14.1.u.nai}
			\includegraphics[width=1.45 in]{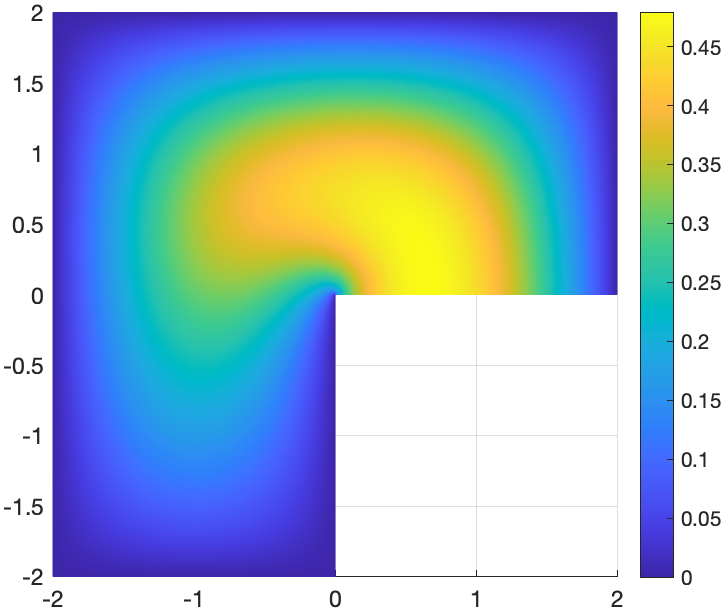}
		}
		\subfloat [$B_4, |u_R-u_h^N|$] {\label{fig.270.14.1.diff.nai}
			\includegraphics[width=1.45 in]{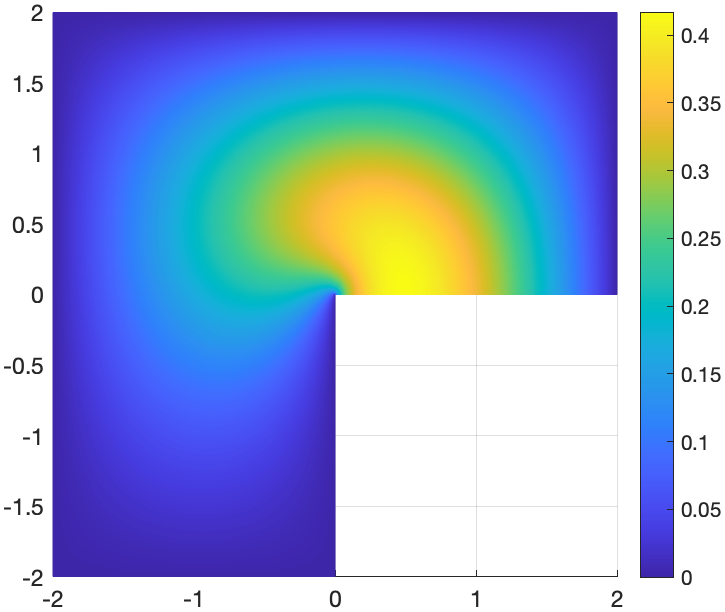}
		}
		\subfloat [$B_4, u_h^M$] {\label{fig.270.14.1.u.mod}
			\includegraphics[width=1.45 in]{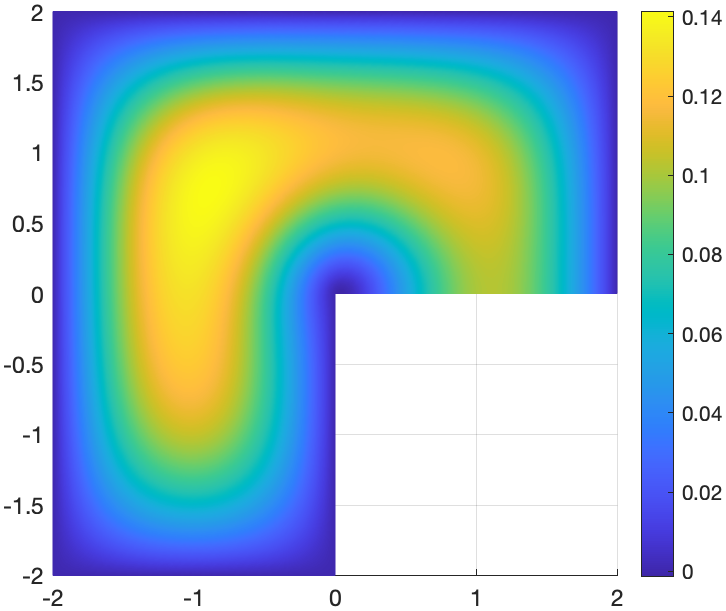}
		}
		\subfloat [$B_4, |u_R-u_h^M|$] {\label{fig.270.14.1.diff.mod}
			\includegraphics[width=1.45 in]{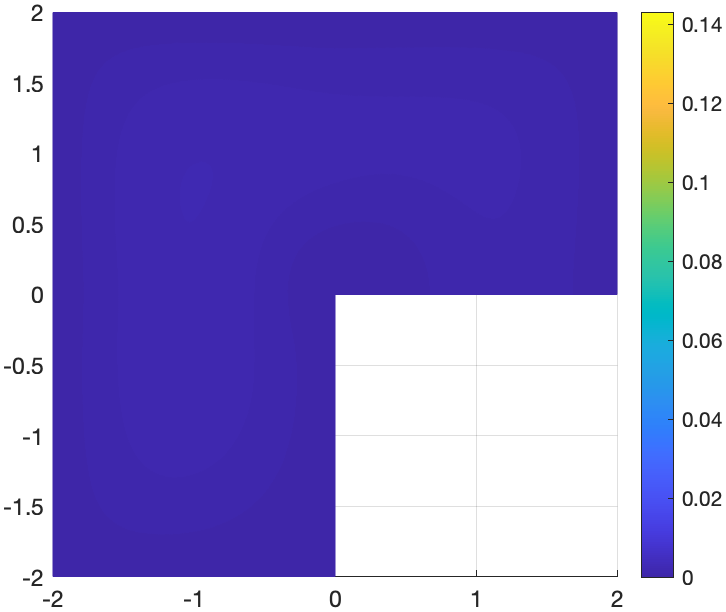}
		}
		\caption{\Cref{exa.6.1.270}: Domain III; Boundary type: $B_3, B_4$; $u_h^N, u_h^M$ and their differences with $u_R$.}
		\label{fig.exa.270.2}
	}
	
	
	
	\begin{table}[h]
		\centering
		\caption{\Cref{exa.6.1.270}: Domain III, $L^\infty$ errors.}\label{ex3err}
		\begin{tabular}{c|c|cccc} 
			\hline
			& Boundary type & 3 & 4 & 5 & 6 \\ 
			\hline
			\multirow{4}{*}{$\|u_R-u_h^N\|_\infty$} 
			& $B_1$ & 1.2837e-01 & 1.3782e-01 & 1.4183e-01 & 1.4305e-01 \\
			& $B_2$ & 1.4597e-02 & 4.5660e-03 & 2.1383e-03 & 1.8196e-03 \\
			& $B_3$ & 3.0181e-01 & 3.5064e-01 & 3.8071e-01 & 3.9915e-01 \\
			& $B_4$ & 3.0180e-01 & 3.5063e-01 & 3.8071e-01 & 3.9914e-01 \\ 
			\hline
			\multirow{4}{*}{$\|u_R-u_h^M\|_\infty$} 
			& $B_1$ & 5.8104e-03 & 3.4371e-03 & 1.6673e-03 & 9.4410e-04 \\
			& $B_2$ & 1.4597e-02 & 4.5660e-03 & 2.1383e-03 & 1.8196e-03 \\
			& $B_3$ & 1.0371e-02 & 4.7030e-03 & 1.9051e-03 & 1.0291e-03 \\
			& $B_4$ & 1.0368e-02 & 4.7012e-03 & 1.9053e-03 & 1.0273e-03 \\
			\hline
		\end{tabular}
	\end{table}
	
	The convergence rates of $w_h$ and $u_h^M$ are shown in \Cref{ex3rate}. From these results, we observe that $u_h^M$ exhibits a convergence rate of order $h^1$, whereas the convergence rate of $w_h$ approaches $h^\alpha$, with $\alpha = 2/3$ for boundary types $B_1$ and $B_2$, and $\alpha = 1/3$ for boundary types $B_3$ and $B_4$. These observations are consistent with the results in \Cref{PoissonError} and \Cref{thmerr}.

	
	\begin{table}[H]
		\centering
		\caption{\Cref{exa.6.1.270}: Domain III, convergence rates of $w_h,u_h^M$.}\label{ex3rate}
		\begin{tabular}{c|c|llllll}
			\hline
			& Boundary type & j=3 & j=4 & j=5 & j=6 & j=7 & j=8 \\ \hline
			\multirow{4}{*}{$\mathcal R$ for $u_h^M$ }  & $B_1$  & 0.84  & 0.92  & 0.97  & 0.99  & 0.99  & 1.00  \\ 
			& $B_2$  & 0.86  & 0.96  & 0.99  & 1.00  & 1.00  & 1.00  \\
			& $B_3$  & 0.83  & 0.92  & 0.97  & 0.99  & 1.00  & 1.00  \\ 
			& $B_4$  & 0.83  & 0.92  & 0.97  & 0.99  & 1.00  & 1.00  \\ \hline
			\multirow{4}{*}{$\mathcal R$ for $w_h$ } & $B_1$  & 0.83  & 0.88  & 0.87  & 0.84  & 0.80  & 0.77  \\ 
			& $B_2$  & 0.86  & 0.96  & 0.99  & 1.00  & 1.00  & 1.00  \\ 
			& $B_3$  & 0.69  & 0.62  & 0.51  & 0.43  & 0.38  & 0.36  \\ 
			& $B_4$  & 0.69  & 0.62  & 0.51  & 0.43  & 0.38  & 0.36  \\ 
			\hline
		\end{tabular}
	\end{table}
	
	
	From \Cref{fig.exa.270.1}, we observe that $|u_R-u_h^N|\rightarrow 0$ as mesh is refined for the boundary type $B_2$, implying that Poisson's problem with $f\equiv C$ (with $C$ being some constant) admits a weak solution $w\in \mathcal{S}$. Therefore, we consider a new source term $f \in L^2(\Omega)$ defined as follows:
	\ali{\label{f-3}
		f=\ca{
			1,\quad &x\geq 0,\, y\geq 0,\\
			0,\quad &x<0,\, y\geq 0,\\
			-1,\quad &y<0.
		}
	}
	We repeat the numerical test, and report the solutions in \Cref{fig.exa.270.3}, from which we observe that the naive mixed finite element solution $u_h^N$ converges to an incorrect solution, while the solution obtained from Algorithm~\ref{algorithm-modified} converges to the true solution. 
	
	
	\fig{[h]
		\centering
		\subfloat [$B_2, u_h^N$] {\label{fig.270.12.2.u.nai}
			\includegraphics[width=1.45 in]{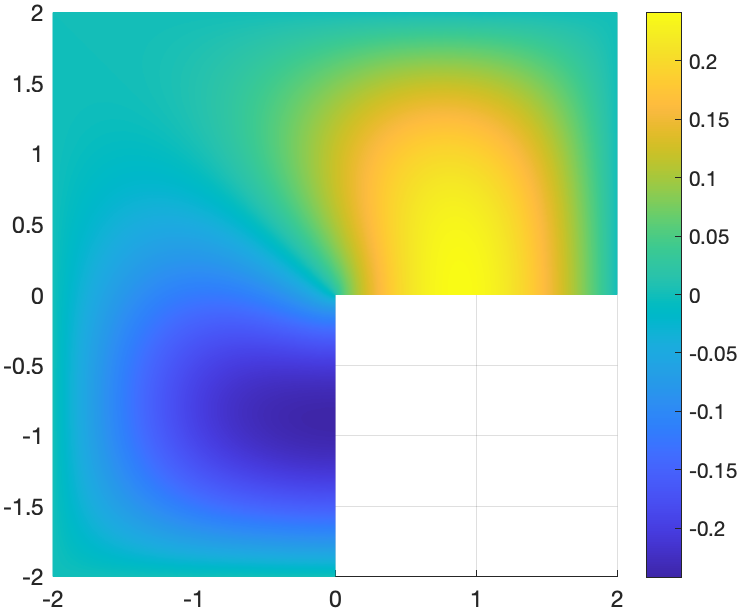}
		}
		\subfloat [$B_2, |u_R-u_h^N|$] {\label{fig.270.12.2.diff.nai}
			\includegraphics[width=1.45 in]{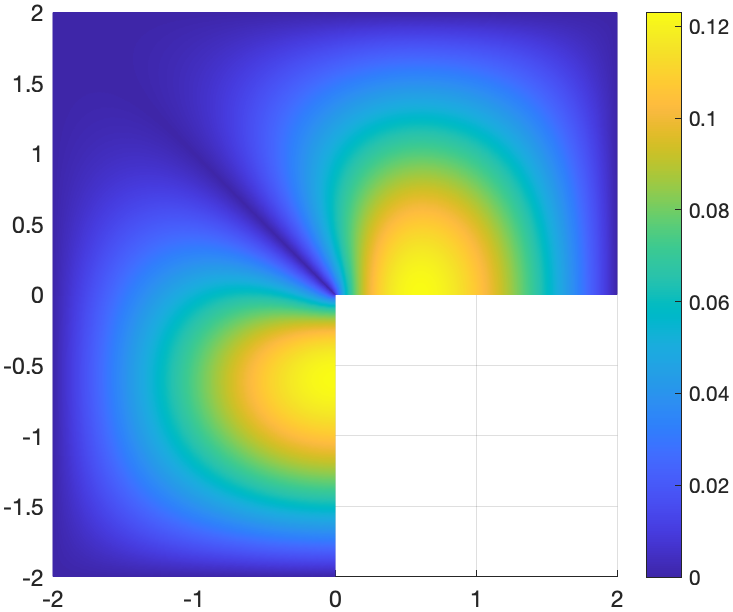}
		}
		\subfloat [$B_2, u_h^M$] {\label{fig.270.12.2.u.mod}
			\includegraphics[width=1.45 in]{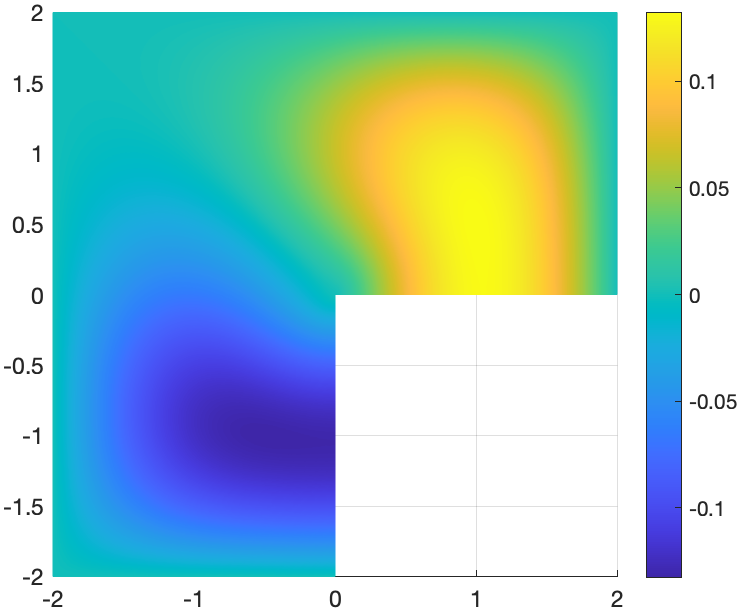}
		}
		\subfloat [$B_2, |u_R-u_h^M|$] {\label{fig.270.12.2.diff.mod}
			\includegraphics[width=1.45 in]{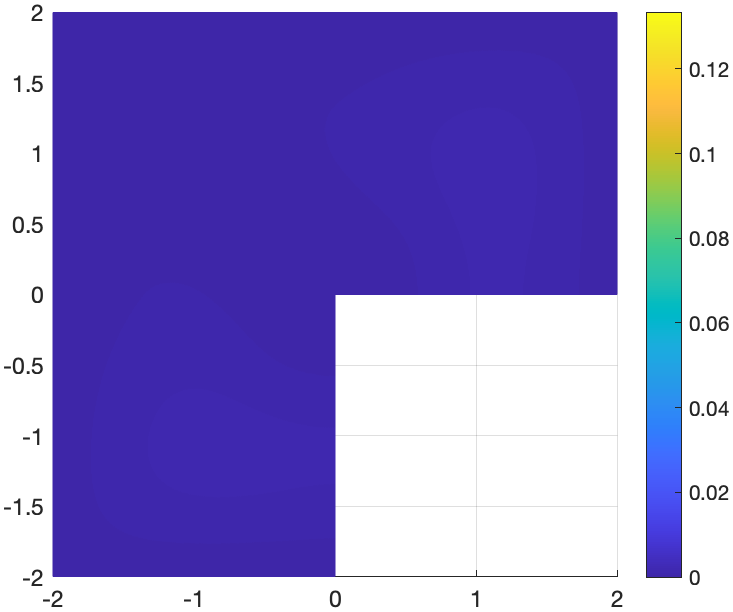}
		}
		\caption{\Cref{exa.6.1.270}: Domain III; Boundary type: $B_2$; $f$ satisfies \eqref{f-3}; $u_h^N, u_h^M$ and differences.}
		\label{fig.exa.270.3}
	}
}
%
%
%
%
%
%
%
%
\exa{\label{exa.6.1.315}
	Consider domain IV, $\Omega=[-2,2]^2\backslash\{(x,y)\mid -x<y<0\}$, with the initial mesh and the mesh after $1$ refinement in \Cref{ex4mesh}. The boundary conditions are specified as $B_3$ and $B_4$. 
	\fig{[hbt!]
		\centering
		\subfloat [Initial mesh] {
			\includegraphics[width=1.1 in]{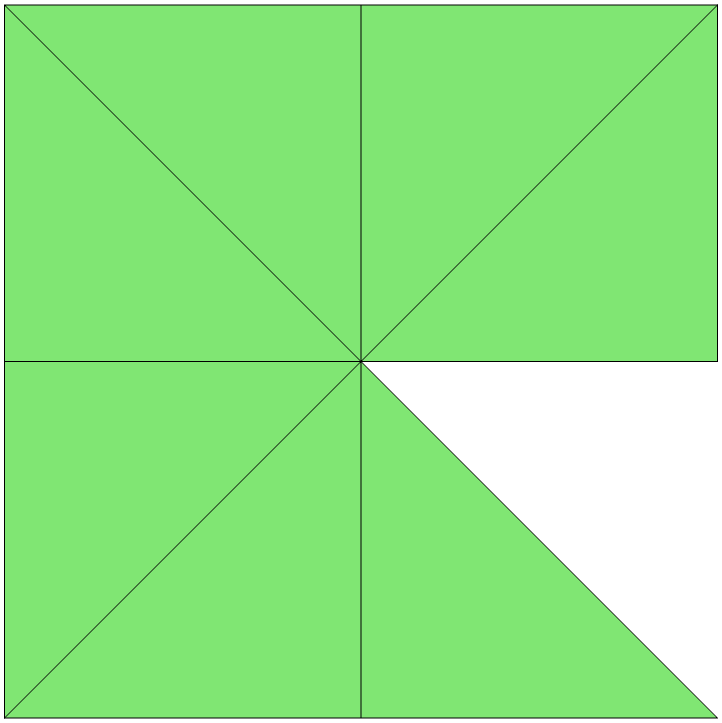}
		}\hspace{4em}
		\subfloat [$1$ refinement] {
			\includegraphics[width=1.1 in]{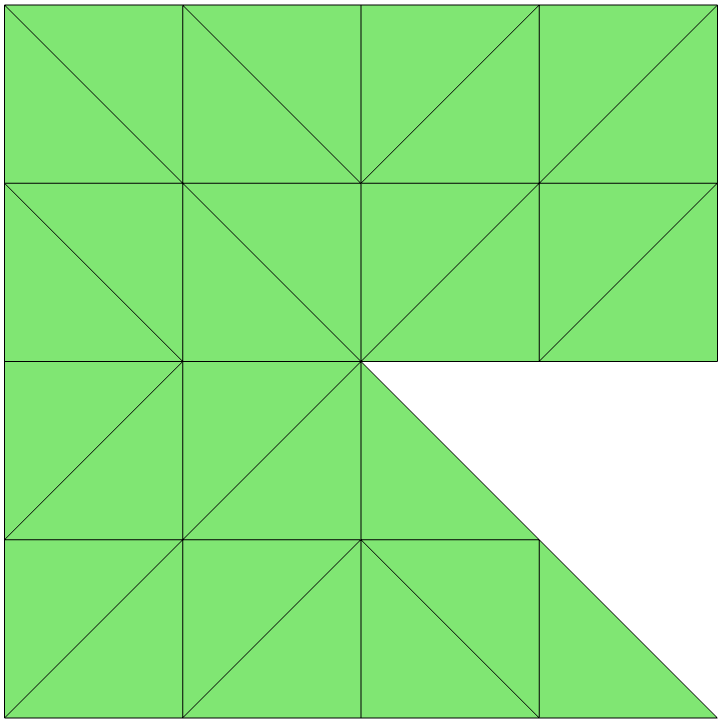}
		}
		\caption{\Cref{exa.6.1.315}: Initial mesh and mesh after $1$ refinement.}\label{ex4mesh}
	}
	
	The largest interior angle is $\frac74\pi$, and $d_\perp = 2$. We consider $f$ given in \eqref{f-3}. 
	We compute the naive mixed finite element solution $u_n^N$; the modified mixed finite element solution $u_h^{M*}$ from Algorithm \ref{algorithm-modified} with $d_\perp=1$; and the modified mixed finite element solution $u_h^{M}$ from Algorithm \ref{algorithm-modified} with $d_\perp=2$. The reference solution is obtained using the $C^0$-IPDG method with penalty parameter $\sigma=68$. 
	
	The numerical solutions $u_h^N$, $u_h^{M*}$, and $u_h^M$, computed on a mesh after eight uniform refinements, along with their differences from the reference solution, are presented in \Cref{fig.exa.315}.     
	The $L^\infty$ norm of differences after 3 to 6 times refinements of the mesh are presented in \Cref{ex4error}.
	From these results, we observe that the finite element solutions $u_h^N$ and $u_h^{M*}$ converge to incorrect solutions, whereas the solution obtained from Algorithm~\ref{algorithm-modified} with $d_\perp=2$ converges to the true solution, thereby confirming the necessity of properly accounting for the dimension of the orthogonal complement in the modified formulation.
	
	\fig{[h]
		\centering
		\subfloat [$B_3, u_h^N$] {\label{fig.315.13.2.u.nai}
			\includegraphics[width=1.45 in]{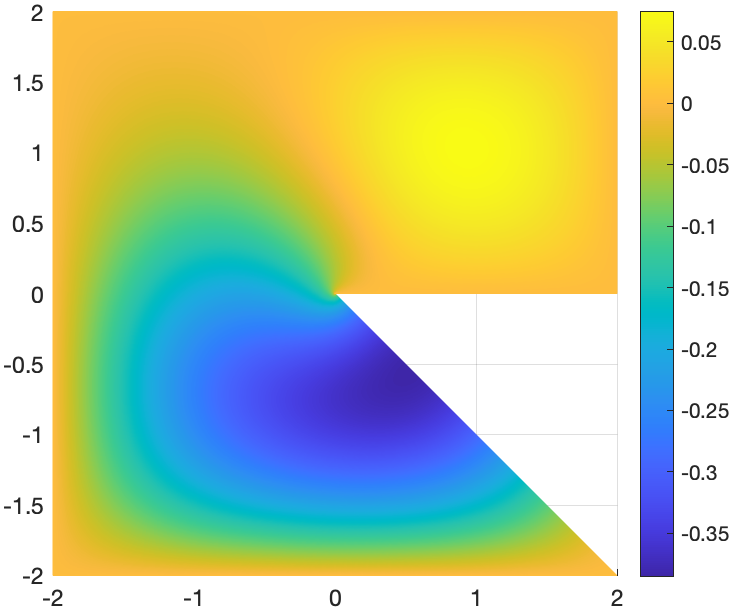}
		}
		\subfloat [$B_3, |u_R-u_h^N|$] {\label{fig.315.13.2.diff.nai}
			\includegraphics[width=1.45 in]{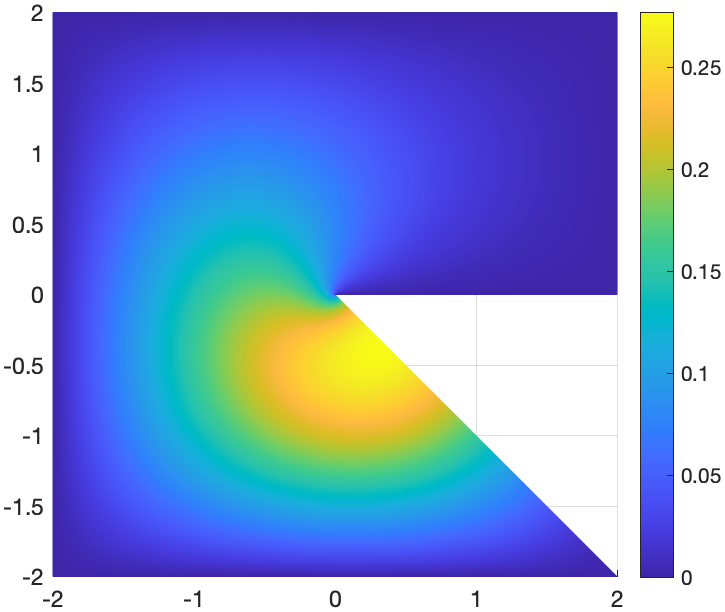}
		}
		\subfloat [$B_3, u_h^{M*}$] {\label{fig.315.13.2.u.modmid}
			\includegraphics[width=1.45 in]{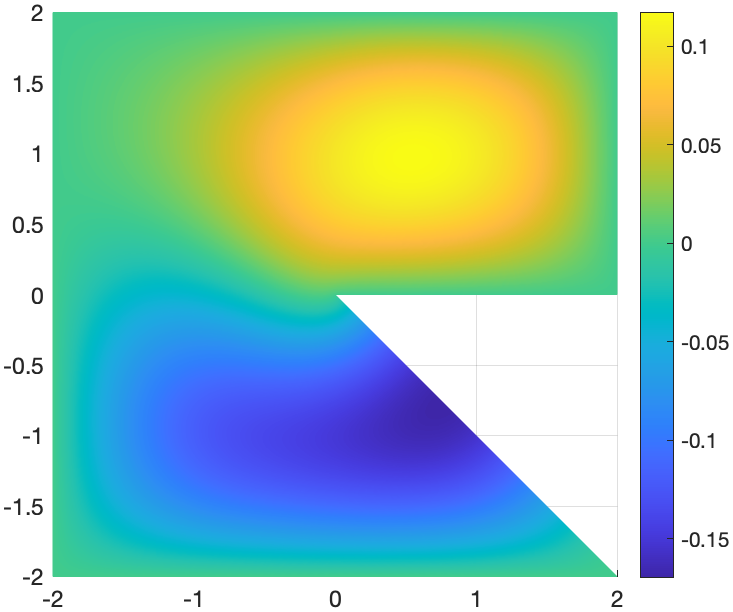}
		}
		\subfloat [$B_3, |u_R-u_h^{M*}|$] {\label{fig.315.13.2.diff.modmid}
			\includegraphics[width=1.45 in]{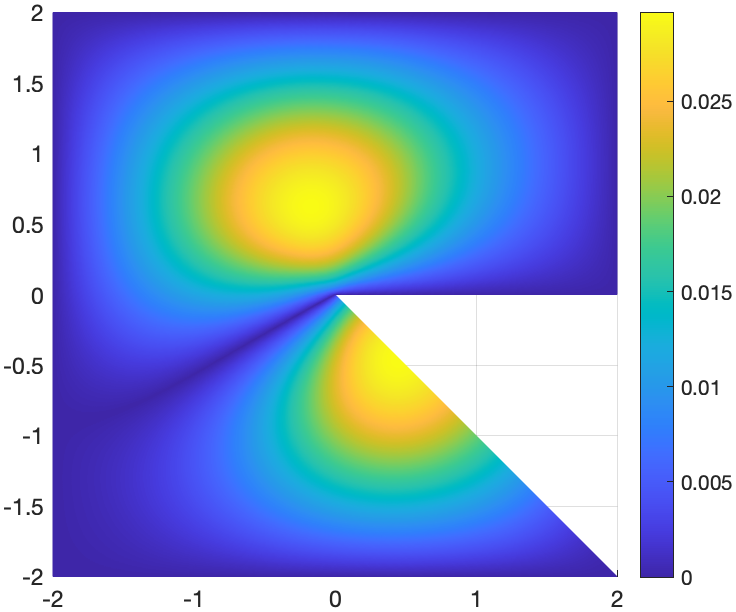}
		}\\ \vspace{1mm}
		\subfloat [$B_3, u_h^M$] {\label{fig.315.13.2.u.mod}
			\includegraphics[width=1.45 in]{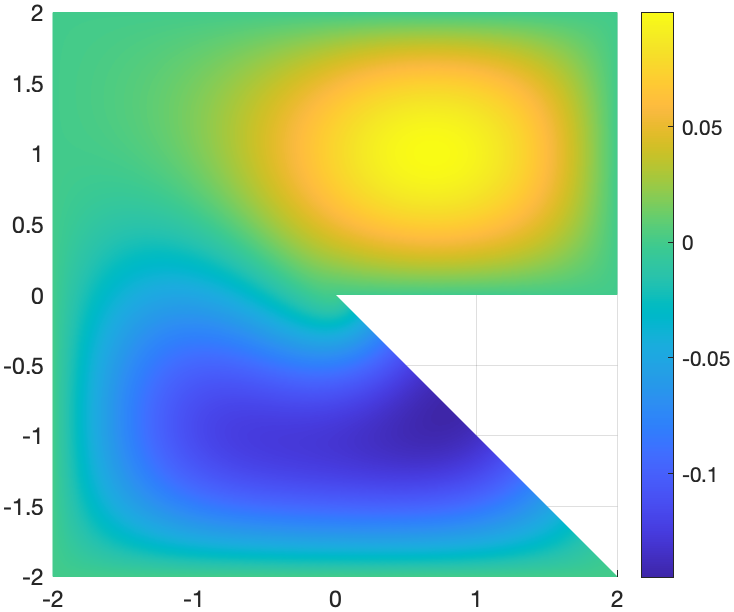}
		}
		\subfloat [$B_3, |u_R-u_h^M|$] {\label{fig.315.13.2.diff.mod}
			\includegraphics[width=1.45 in]{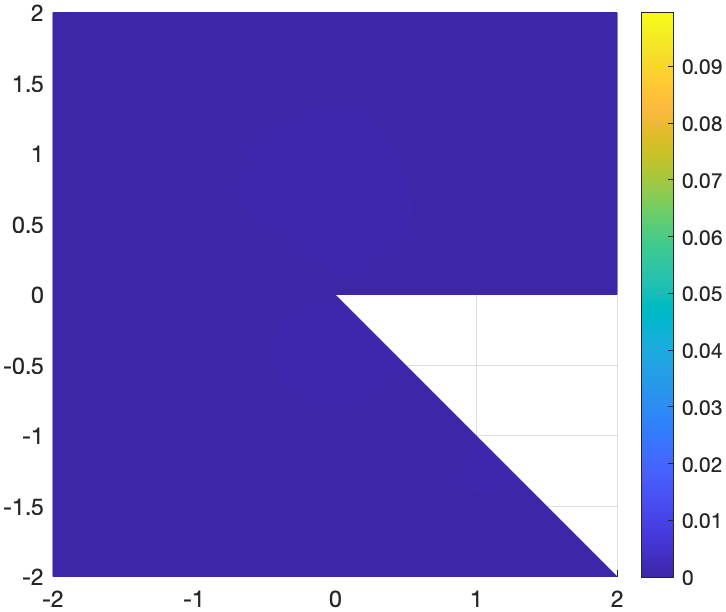}
		}
		\subfloat [$B_4, u_h^N$] {\label{fig.315.14.2.u.nai}
			\includegraphics[width=1.45 in]{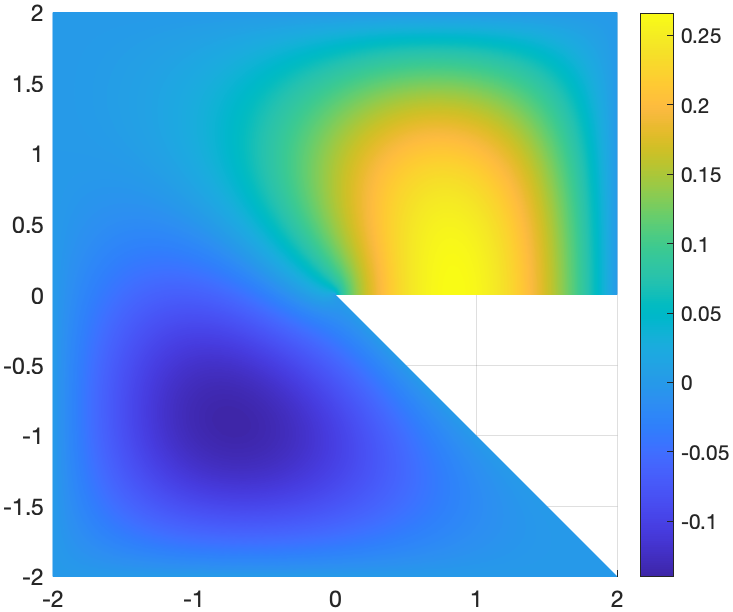}
		}
		\subfloat [$B_4, |u_R-u_h^N|$] {\label{fig.315.14.2.diff.nai}
			\includegraphics[width=1.45 in]{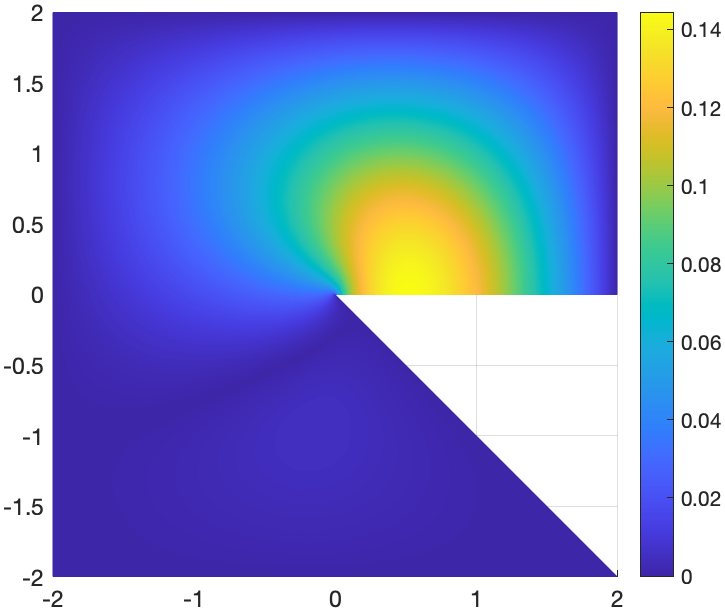}
		}\\ \vspace{1mm}
		\subfloat [$B_4, u_h^{M*}$] {\label{fig.315.14.2.u.modmid}
			\includegraphics[width=1.45 in]{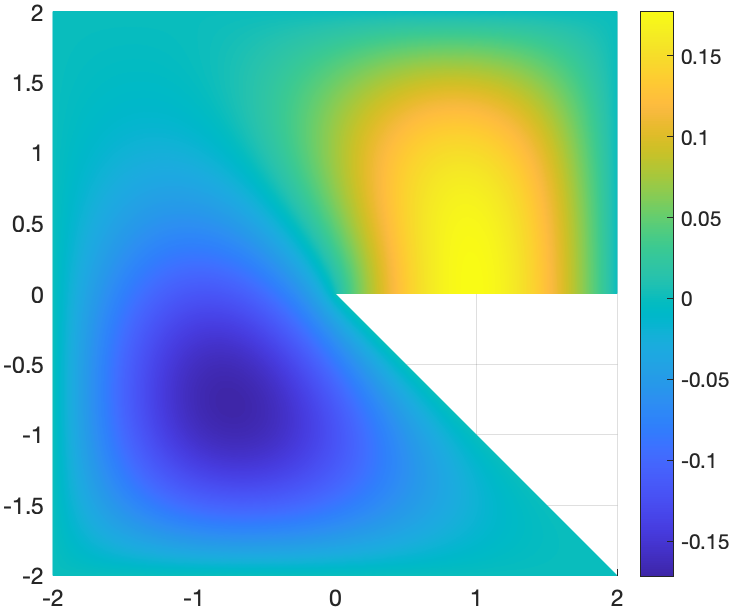}
		}
		\subfloat [$B_4, |u_R-u_h^{M*}|$] {\label{fig.315.14.2.diff.modmid}
			\includegraphics[width=1.45 in]{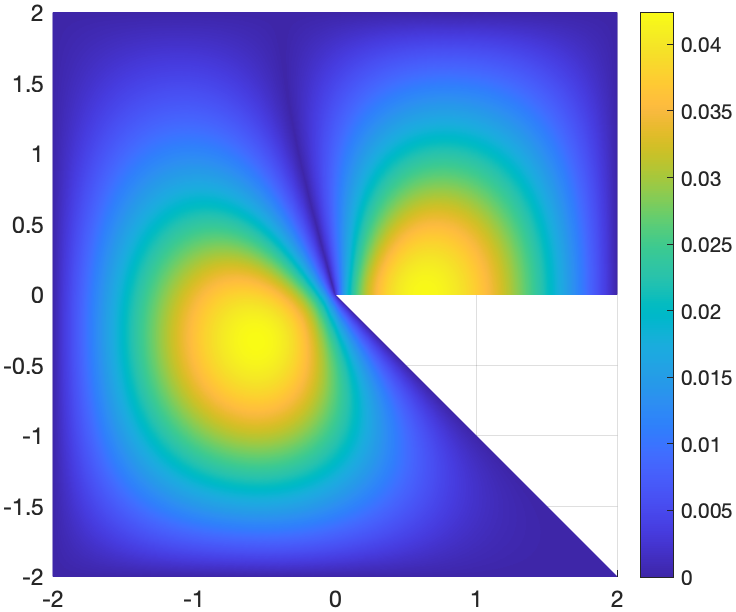}
		}
		\subfloat [$B_4, u_h^M$] {\label{fig.315.14.2.u.mod}
			\includegraphics[width=1.45 in]{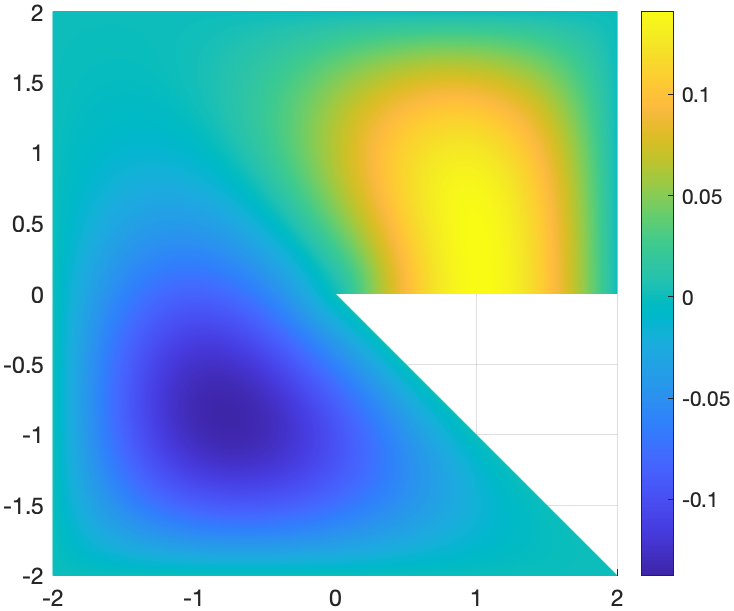}
		}
		\subfloat [$B_4, |u_R-u_h^M|$] {\label{fig.315.14.2.diff.mod}
			\includegraphics[width=1.45 in]{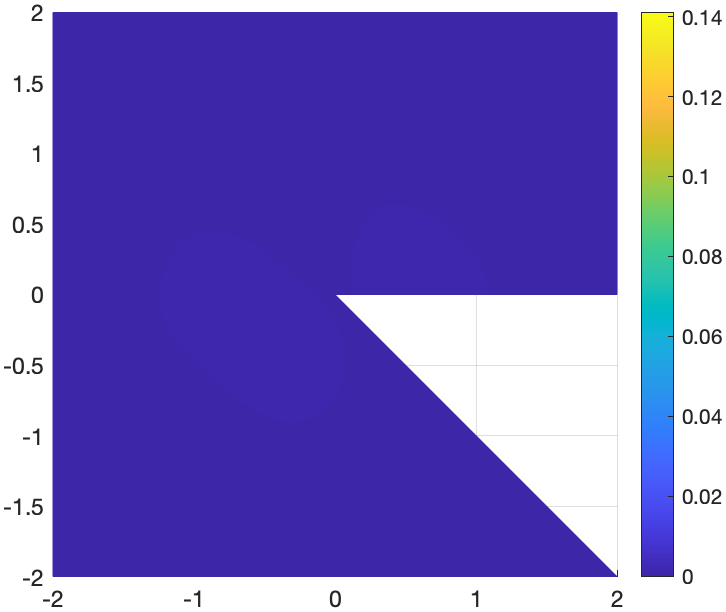}
		}
		\caption{\Cref{exa.6.1.315}: Domain IV; Boundary type: $B_3, B_4$; $u_h^N, u_h^{M*}, u_h^M$ and differences.}
		\label{fig.exa.315}
	}
	
	
	
	\begin{table}[h]
		\centering
		\caption{\Cref{exa.6.1.315}: Domain IV, $L^\infty$ errors. }\label{ex4error}
		\begin{tabular}{c|c|cccc} 
			\hline
			& Boundary type & 3 & 4 & 5 & 6 \\ 
			\hline
			\multirow{2}{*}{$\|u_R-u_h^N\|_\infty$} 
			& $B_3$ & 1.9797e-01 & 2.3068e-01 & 2.5019e-01 & 2.6279e-01 \\
			& $B_4$ & 1.0727e-01 & 1.2388e-01 & 1.3283e-01 & 1.3816e-01 \\ 
			\hline
			\multirow{2}{*}{$\|u_R-u_h^{M*}\|_\infty$} 
			& $B_3$ & 2.8418e-02 & 2.9444e-02 & 2.9573e-02 & 2.9527e-02 \\
			& $B_4$ & 3.8119e-02 & 4.1158e-02 & 4.1791e-02 & 4.2073e-02 \\
			\hline
			\multirow{2}{*}{$\|u_R-u_h^M\|_\infty$} 
			& $B_3$ & 6.4356e-03 & 3.1130e-03 & 1.2959e-03 & 7.3513e-04 \\
			& $B_4$ & 6.9425e-03 & 2.2060e-03 & 9.6986e-04 & 7.5067e-04 \\
			\hline
		\end{tabular}
	\end{table}
	
	The convergence rates of $w_h$ and $u_h^M$ are shown in \Cref{ex4rate}. From these results, we observe that $u_h^M$ exhibits a convergence rate of order $h^1$, whereas the convergence rate of $w_h$ approaches $h^\alpha$ with $\alpha=2/7$ under boundary types $B_3$ and $B_4$.
	
	\begin{table}[H]
		\centering
		\caption{\Cref{exa.6.1.315}: Domain IV, convergence rates of $w_h,u_h^M$.}\label{ex4rate}
		\begin{tabular}{c|c|llllll}
			\hline
			& Boundary type & j=3 & j=4 & j=5 & j=6 & j=7 & j=8 \\ \hline
			\multirow{2}{*}{$\mathcal R$ for $u_h^M$ } 
			& $B_3$  & 0.83  & 0.92  & 0.97  & 0.99  & 1.00  & 1.00  \\ 
			& $B_4$  & 0.83  & 0.92  & 0.97  & 0.99  & 1.00  & 1.00  \\ \hline
			\multirow{2}{*}{$\mathcal R$ for $w_h$ } 
			& $B_3$  & 0.67  & 0.58  & 0.47  & 0.38  & 0.33  & 0.31  \\ 
			& $B_4$  & 0.67  & 0.57  & 0.46  & 0.38  & 0.33  & 0.31  \\ 
			\hline
		\end{tabular}
	\end{table}
	
}


\subsection{Neumann boundary conditions}
We consider the Neumann boundary conditions discussed in \Cref{sec-neumann}.

\exa{\label{exa.6.2.270}
	
	Consider domain III, $\Omega=[-2,2]^2\backslash\{(x,y)|x>0,y<0\}$, with boundary type $B_5$. 
	We take $f$ as \eqref{f-3}, which satisfies the compatibility condition $\int_\Omega f \,\mathrm dx = 0$. 
	For the reference solution, the mesh-dependent problem requires the functions in $V_h$ to vanish at the origin. The interior penalty parameter is set to $\sigma=39$.
	The triangulation and the mesh after $1$ refinement are the same as in \Cref{exa.6.1.270}. 
	For this polygonal domain, there exists an interior angle of $\frac32\pi$, and $d_\perp = 1$. 
	The numerical solutions $u_h^N$ and $u_h^M$, computed on a mesh after eight uniform refinements, along with their differences from the reference solution, are presented in \Cref{fig.exa.270.neumann}. 
	The $L^\infty$ norm of the differences $|u_R-u_h^N|$, $|u_R-u_h^M|$ after $3$ to $6$ refinements are presented in \Cref{ex5errror}.
	From these results, we observe that the naive mixed finite element solution $u_h^N$ converges to an incorrect solution, whereas the solution obtained from Algorithm~\ref{algorithm-modified} with $V_h$ replaced by $\bar V_h$ in \eqref{barVhspace} converges to the true solution. 
	
	\fig{[h]
		\centering
		\subfloat [$B_5, u_h^N$] {\label{fig.270.2.2.u.nai}
			\includegraphics[width=1.45 in]{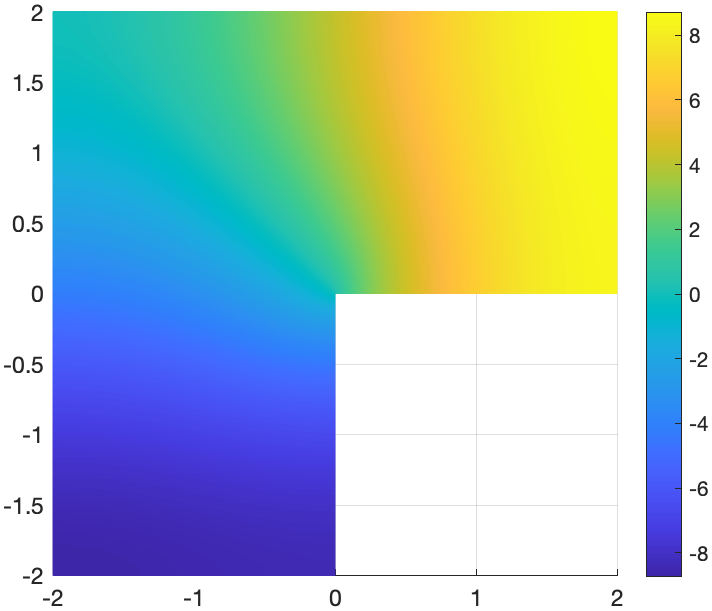}
		}
		\subfloat [$B_5, |u_R-u_h^N|$] {\label{fig.270.2.2.diff.nai}
			\includegraphics[width=1.45 in]{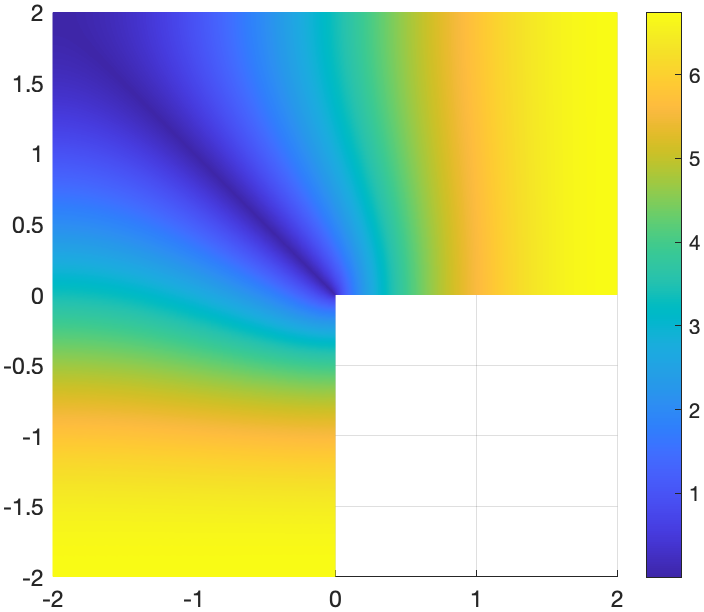}
		}
		\subfloat [$B_5, u_h^M$] {\label{fig.270.2.2.u.mod}
			\includegraphics[width=1.45 in]{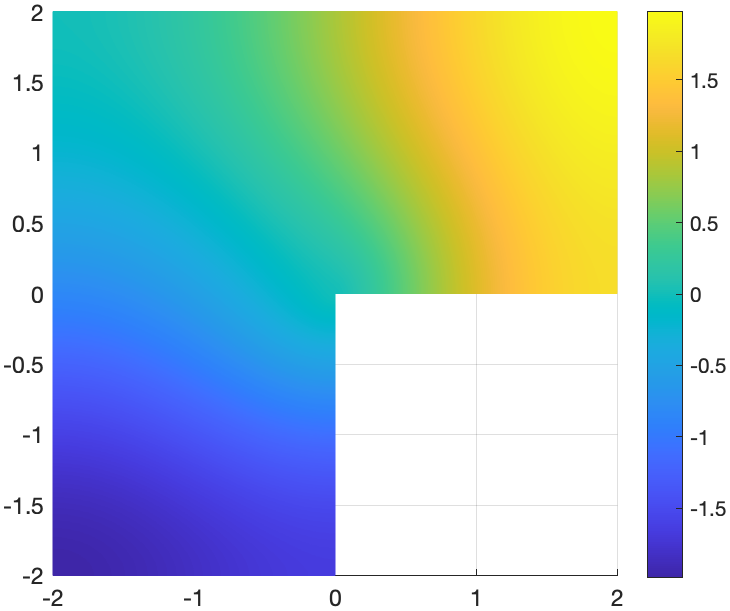}
		}
		\subfloat [$B_5, |u_R-u_h^M|$] {\label{fig.270.2.2.diff.mod}
			\includegraphics[width=1.45 in]{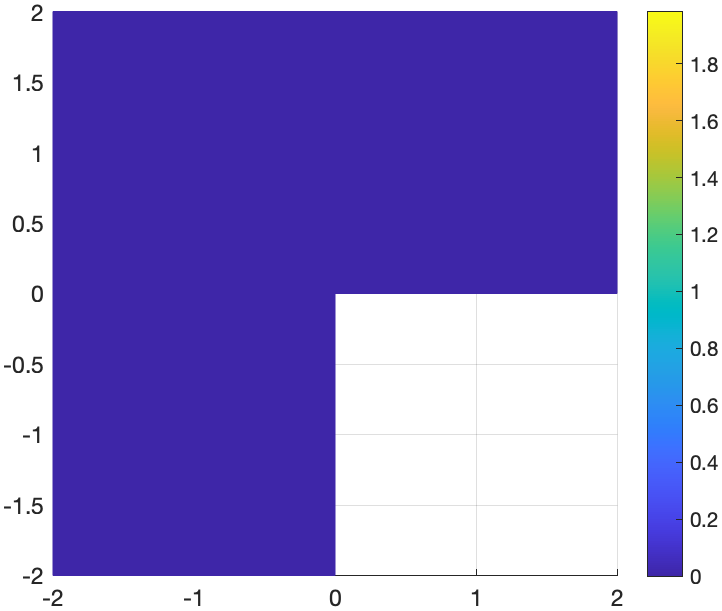}
		}
		\caption{\Cref{exa.6.2.270}: Domain III; Boundary type: $B_5$; $f$ satisfies \eqref{f-3}; $u_h^N, u_h^M$ and differences.}
		\label{fig.exa.270.neumann}
	}
	
	
	
	\begin{table}[h]
		\centering
		\caption{\Cref{exa.6.2.270}: Domain III, $L^\infty$ errors.}\label{ex5errror}
		\begin{tabular}{c|c|cccc} 
			\hline
			& Boundary type & 3 & 4 & 5 & 6 \\ 
			\hline
			\multirow{1}{*}{$\|u_R-u_h^N\|_\infty$} 
			& $B_5$ & 6.5195e+00 & 6.6456e+00 & 6.6964e+00 & 6.7209e+00 \\ 
			\hline
			\multirow{1}{*}{$\|u_R-u_h^M\|_\infty$} 
			& $B_5$ & 3.8229e-02 & 1.7219e-02 & 8.0944e-03 & 6.4859e-03 \\
			\hline
		\end{tabular}
	\end{table}
	
	The convergence rates of $w_h$ and $u_h^M$ are shown in \Cref{ex5rate}. From these results, we observe that $u_h^M$ exhibits a convergence rate of order $h^1$, whereas the convergence rate of $w_h$ approaches $h^\alpha$ with $\alpha = 2/3$ under boundary type $B_5$, which are consistent with the result in \Cref{neumannrate}.
	
	\begin{table}[H]
		\centering
		\caption{\Cref{exa.6.2.270}: Domain III, convergence rates of $w_h,u_h^M$.}\label{ex5rate}
		\begin{tabular}{c|c|llllll}
			\hline
			& Boundary type & j=3 & j=4 & j=5 & j=6 & j=7 & j=8 \\ \hline
			\multirow{1}{*}{$\mathcal R$ for $u_h^M$ }  & $B_5$  & 0.85  & 0.96  & 0.96  & 0.98  & 0.99  & 0.99  \\ 
			\hline
			\multirow{1}{*}{$\mathcal R$ for $w_h$ } & $B_5$ & 0.73  & 0.73  & 0.72  & 0.70  & 0.69  & 0.68  \\
			\hline
		\end{tabular}
	\end{table}
	
}

\section*{Acknowledgments}
H. Li was supported in part by the National Science Foundation Grant DMS-2208321 and a Japan Society for the Promotion of Science Research Grant.
N. Yi was supported by the National Key R \& D Program of China (2024YFA1012600) and NSFC (12431014).
P. Yin’s research was supported by the University of Texas at El Paso Startup Award.


\bibliographystyle{acm}
\bibliography{bib/ref}

\end{CJK}
\end{document}